\documentclass[12pt]{article}
\usepackage{amssymb}
\usepackage{latexsym}
\usepackage{graphicx}
\usepackage{amsmath}
\usepackage{subfigure}
\usepackage{epstopdf}
\usepackage[linesnumbered, ruled]{algorithm2e}

\newtheorem{theorem}{Theorem}

\newtheorem{corollary}[theorem]{Corollary}

\newtheorem{lemma}[theorem]{Lemma}

\newtheorem{proposition}[theorem]{Proposition}

\newenvironment{proof}[1][Proof]{\noindent\textbf{#1.} }{\ \rule{0.5em}{0.5em}}
\def\lab(#1)#2{\put(#1){\makebox(0,0)[c]{#2}}}

\begin{document}

\title{An exact algorithm for the bottleneck 2-connected $k$-Steiner network problem in $L_p$ planes}
\author{M. Brazil\and C.J. Ras \and D.A. Thomas}

\maketitle

\begin{abstract}
We present the first exact polynomial time algorithm for constructing optimal geometric bottleneck $2$-connected Steiner networks containing at most $k$ Steiner points,
where $k>2$ is a constant. Given a set of $n$ vertices embedded in an $L_p$ plane, the objective of the problem is to find a $2$-connected network, spanning the given vertices and at most $k$ additional vertices, such that the length of the longest edge is minimised. In contrast to the discrete version of this problem the additional vertices may be located anywhere in the plane. The problem is motivated by the modelling of relay-augmentation for the
optimisation of energy consumption in wireless ad hoc networks. Our algorithm employs Voronoi diagrams and properties of block-cut-vertex decompositions of graphs to find an optimal solution in $O(n^k\log^{\frac{5k}{2}}n)$ steps when $1<p<\infty$ and in $O(n^2\log^{\frac{7k}{2}+1}n)$ steps when $p\in\{1,\infty\}$.
\end{abstract}

\section{Introduction}
Reducing energy consumption due to data transmission is a primary concern when designing wireless radio networks, since, especially in the case of autonomous ad hoc networks such as sensor networks, node failure due to battery depletion must be postponed for as long as possible. Generally, maximum transmission power is utilised at the nodes communicating across the \textit{bottleneck} (or longest link) of the network. In ad hoc networks, the process of \textit{relay-augmentation} has proven to be effective at optimising the bottleneck length \cite{bae1,brazil}.

Given a set of transmitters in the plane, the primary goal of relay augmentation is to construct a network of minimum bottleneck length by introducing new transmitters (relays) and links. The resultant network (\textit{augmented network}) must satisfy a given connectivity constraint and may contain at most a bounded number of relays. The connectivity constraint stipulates the minimum number of transmitters (including relays) that may fail before the network becomes disconnected. An upper bound on the number of relays is not only realistic in practice, but is also necessary for guaranteeing that a solution to the relay-augmentation problem exists, since, in the limit one can always reduce the bottleneck length by deploying an extra relay.

\begin{figure}[htb]
  \begin{center}
    \includegraphics[scale=0.6]{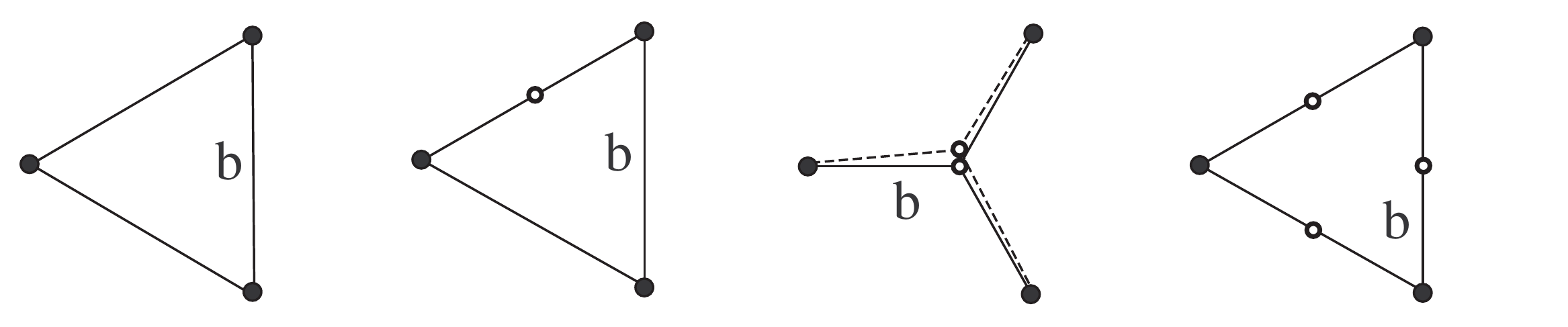}\\
  \end{center}
  \caption{Optimal solutions when $n=3$, $c=2$, and $k=0,1,2,3$}
  \label{figEgs1}
\end{figure}

An appropriate model for relay augmentation is the \textit{bottleneck $c$-connected $k$-Steiner network problem}: the given transmitters are represented by vertices called \textit{terminals} embedded in an $L_p$ plane, relays are represented by variable \textit{Steiner points} in the plane, $k$ is the above-mentioned bound on the number of relays, and $c$ is the connectivity constraint. Figure \ref {figEgs1} illustrates optimal solutions to this problem for four cases of three equilaterally positioned terminals (filled circles) in the Euclidean plane when $c=2$ and $k=0,1,2,3$ respectively. The third subfigure depicts a solution where both Steiner points (open circles) occupy the same location at the centre of the triangle. In each case an arbitrarily selected (from equally long edges) bottleneck edge is labelled by the letter `b'.

The bottleneck $1$-connected $k$-Steiner network problem is NP-hard in both the Euclidean and rectilinear planes when $k$ is part of the input \cite{sarr,wang}; and also in general metrics when the number of Steiner points is not explicitly bounded, but the minimum degree of any Steiner point is at least $3$ (see \cite{berm}). There exist exact polynomial time algorithms for constructing optimal $1$-connected augmented networks when $k$ is constant (i.e., not part of the input); see \cite{bae1,brazil}. No complexity results are to be found in the literature for $c>1$.

In practice, $1$-connectivity is often insufficient. Practical networks require a degree of survivability against the inevitable disruption or failure of nodes or links. On the other hand, for most networks $2$-connectivity is sufficient to provide engineers with enough confidence that the network will not disconnect within the time-period between node failure and subsequent node replacement \cite{mon}. In general, $2$-connectivity is therefore the most cost-efficient and popular option. The mathematical literature most closely related to the survivability aspect of the problem we study here deals with the construction of so called \textit{bottleneck biconnected spanning subgraphs} (see
\cite{mschang,park}); this problem is motivated by the search for heuristics for the bottleneck Travelling Salesman problem.

In this paper we describe a polynomial time algorithm (which we refer to as the $2$-\textit{Bottleneck algorithm}) for solving the bottleneck $2$-connected $k$-Steiner network problem in $L_p$ planes, when $k$ is constant. This may be viewed as a generalisation of both the $c=1$ case as solved by Bae et al. \cite{bae1} and Brazil et al. \cite{brazil}; and of the $c=2$ and $k\leq 2$ case solved in \cite{brazil2}. We rely on an essential geometric component of Bae et al.'s algorithm utilising \textit{farthest colour Voronoi diagrams}, although it is a non-trivial fact that their method extends to the $c=2$ case. Also, in Brazil et al. \cite{brazil2} a process employing binary search is developed to solve one of their subcases. A substantial part of this paper involves a generalisation of this process to $k>2$ Steiner points.

\section{Overview of the $2$-Bottleneck algorithm}\label{over}
Before presenting a broad overview of our algorithm for constructing optimal solutions to the bottleneck $2$-connected $k$-Steiner network problem, we present some basic terminology and then formally define the problem. A general graph concept that is ubiquitous in this paper is that of a \textit{topology}; the topology of a graph is equivalent to the adjacency matrix of its vertices. In a \textit{geometric graph} all vertices have coordinates and there exists a geometric description for each curve representing an edge. In this paper, edges of geometric graphs are always geodesics. Edges incident to a Steiner point are called \textit{Steiner edges}; all other edges are called \textit{terminal edges}. Some of the graphs we discuss are \textit{partly geometric}, in the sense that the terminals have pre-assigned coordinates but the Steiner points do not. This concept is, in fact, central to the construction of optimal bottleneck Steiner networks.

For any graph $G$ in the plane we denote the length of the longest edge of $G$ (with respect to some $L_p$ metric) by $\ell_{\mathrm{max}}(G)$. Let $X$ be a set of vertices (called \textit{terminals}) embedded in $\mathbb{R}^2$.

\noindent\textbf{Definition.} \textit{The \textbf{bottleneck $c$-connected $k$-Steiner network problem} requires the construction of a $c$-connected network $N^*$ spanning both $X$
and a set $S$ of at most $k$ Steiner points, such that $\ell_{\mathrm{max}}(N^*)$ is a minimum across all such networks. The variables are $S$ and the topology of the network.}

An optimal solution to the problem is called a \textit{globally optimal network}. If $Z$ is any set of $2$-connected networks spanning $X$ and at most $k$ Steiner points, then any $N\in Z$ such that $\ell_{\mathrm{max}}(N)\leq \ell_{\mathrm{max}}(N')$ for all $N'\in Z$ is called a \textit{locally optimal network with respect to} $Z$; we also refer to $N$ as a \textit{cheapest network in} $Z$. Similarly, the topology of a globally (locally) optimal network is referred as a \textit{globally (locally) optimal topology}. In this paper we focus on the case $c=2$ with constant $k\geq 3$. We also assume throughout that $|X|=n\geq 2$.

In broad terms, for a given set $X$ of terminals, our algorithm constructs every possible $2$-connected \textit{candidate} topology spanning both $X$ and a set of at most $k$ variable Steiner points. The number of these partly geometric topologies is super-exponential in $|X|$, however, we show how to reduce this to polynomial order. For each candidate topology $N$, coordinates are then assigned to the Steiner points in such a way that the resultant graph has the shortest bottleneck amongst all geometric graphs (on $X$) with topology $N$; this step requires the use of farthest colour Voronoi diagrams (as in \cite{bae1}). Among all resultant graphs, one with the shortest bottleneck is picked as the optimal solution.

We reduce the complexity of the above algorithmic framework further by strategically dividing the set of all candidate topologies into a small number of different \textit{types}. For each type $G(\mathcal{Q})$ we then show that there exists a fast procedure for constructing a geometric graph $N$ with topology of type $G(\mathcal{Q})$, such that $N$ has a shortest bottleneck amongst all geometric graphs with topology of type $G(\mathcal{Q})$. As before, a resultant geometric graph with shortest bottleneck is selected as the globally optimal solution. The division into types is based on a process described in \cite{bae1}, where so called ``abstract topologies" play a similar role. Describing \textit{types} for the bottleneck $2$-connected $k$-Steiner network problem is significantly more complex than for the $1$-connected problem studied by Bae et al., and we therefore defer a detailed discussion to a later section. At this juncture we provide only the following simple illustrative example of a candidate type in the Euclidean plane, and show how the concept is used in this case to optimally locate a Steiner point.

\begin{figure}[h!]
  \begin{center}
    \includegraphics[scale=0.55]{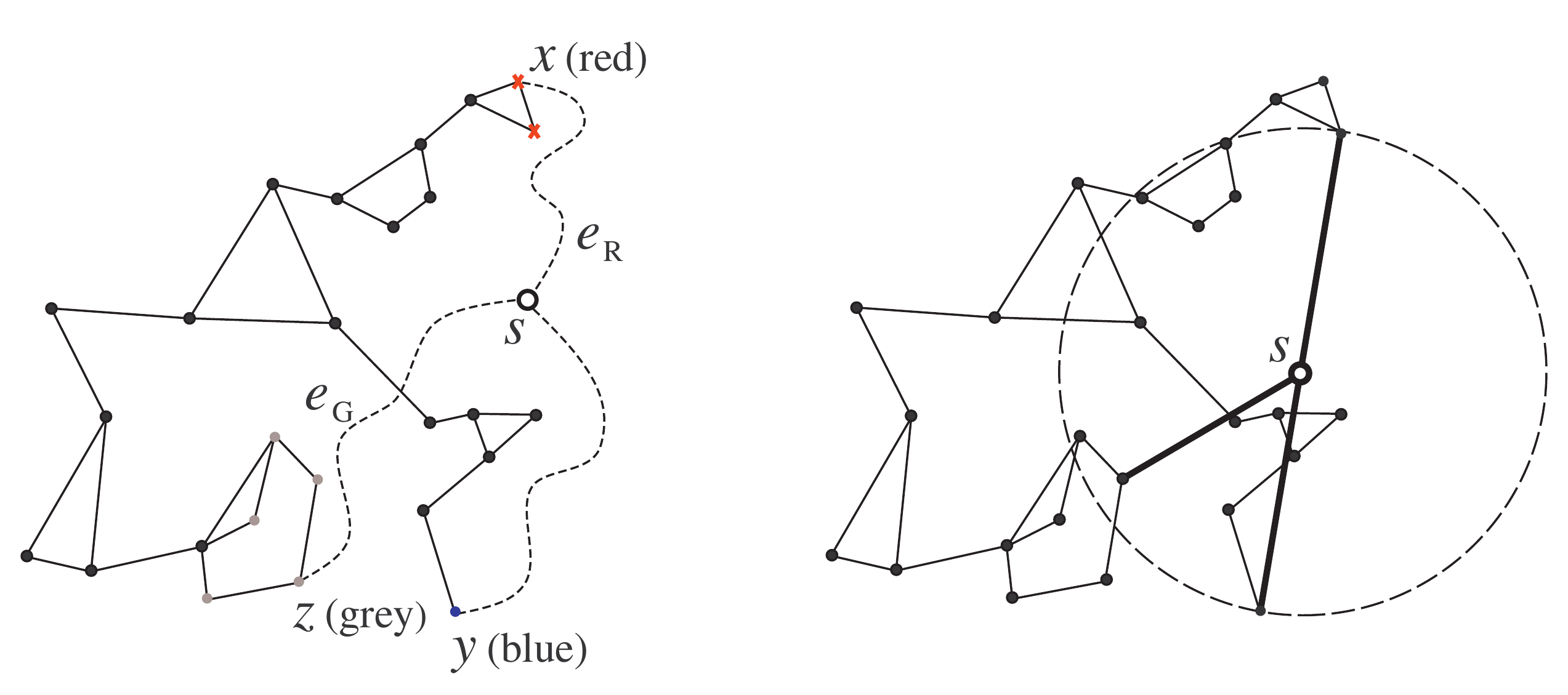}\\
  \end{center}
  \caption{A topology with vertices coloured to produce a single candidate type. A smallest colour-spanning disk locates $s$ optimally for this type}
  \label{figEgs3}
\end{figure}

Figure \ref{figEgs3} shows a candidate topology spanning a single Steiner point $s$ and twenty six terminals. Dashed curves depict Steiner edges incident to a variable Steiner point; and Steiner edges in the resultant graph are represented by bold edges. Terminal $x$ and one of its neighbours are coloured red (these two vertices are represented by crosses); terminal $y$ is coloured blue; and $z$, along with four other specific vertices close to $z$, are coloured grey. Notice firstly that if we replace $z$ as an endpoint of edge $e_\mathrm{G}$ with any grey vertex, and replace $x$ as an endpoint of edge $e_\mathrm{R}$ with any red vertex, the resultant topology is still $2$-connected. In fact, any ($2$-connected) candidate topology ${G}$ on these twenty six terminals, with a single Steiner point, and that contains exactly the same terminal edges as Figure \ref{figEgs3}, must contain the edges $sx',sy$, and $sz'$, where $x'$ is some red vertex and $z'$ is some grey vertex; any other edges incident to $s$ can be deleted from ${G}$ without reducing connectivity. The ten topologies (five grey vertex combinations multiplied by two red vertex combinations) that have these specific properties are said to be of the same \textit{candidate type}. Observe that the candidate type is defined with respect to the graph induced by the terminal edges; this graph is called the \textit{underlying network} and there exists a simple procedure, described later, for constructing it. The sets of coloured terminals are referred to as \textit{valid subsets} of terminals. An optimal location for $s$ with respect to all topologies of this type can now be found by constructing the centre of a smallest disk spanning at least one vertex of each colour. A \textit{smallest colour-spanning disk} is depicted in the second subfigure of Figure \ref{figEgs3}, along with the resultant $2$-connected geometric graph. Constructing smallest colour-spanning disks for more complex examples, such as when there are at least two adjacent Steiner points in the candidate topology, requires the use of \textit{farthest colour Voronoi diagrams}, as recognised by Bae et al. \cite{bae1}.

In Section \ref{sec1} we provide preliminary results and terminology related to connectivity and block-cut-vertex decompositions. Section \ref{types} discusses the construction of candidate types, and Section \ref{final} presents the $2$-Bottleneck algorithm and proves its correctness.


\section{Preliminaries}\label{sec1}
Throughout this paper we only consider finite, simple, and undirected graphs. A graph $G=\langle V(G),E(G)\rangle$ is \textit{connected} if there exists a path
connecting any pair of vertices in $G$. A \textit{component} is a maximal (by inclusion) connected subgraph. The following expressions will have their obvious meanings in this paper: $G-D$, where $D$ is a set of vertices or edges of $G$; $G+E$, where $E$ is a set of edges not in $G$; and $G\pm e$, where $e$ is an edge. A \textit{cut-set} $A$ of
$G$ is any set of vertices such that $G-A$ has strictly more components than $G$; if $|A|=1$ then $A$ is a \textit{cut-vertex}. An \textit{edge-cut} $E$ of
$G$ is any set of edges such that $G-E$ has strictly more components than $G$; if $|E|=1$ then $E$ is a \textit{bridge}.

The \textit{vertex-connectivity} or simply \textit{connectivity} $c=c(G)$ of a graph $G$ is the minimum number of vertices whose removal results
in a disconnected or trivial graph. Therefore $c$ is the minimum cardinality of a cut-set of $G$ if $G$ is connected but not complete; $c=0$ if
$G$ is disconnected; and $c=n-1$ if $G=K_n$, where $K_n$ is the complete graph on $n$ vertices. A graph $G$ is said to be
$c^\prime$-\textit{connected} if $c\geq c^\prime$ for some non-negative integer $c^\prime$. As is standard in the literature, we make an exception for the connectivity definitions of $K_1,K_2$: we assume that $c(K_1)=c(K_2)=2$.

\subsection{Block-cut forests}


This paper makes use of a well-known decomposition process by which any given graph is transformed into a forest. Roughly speaking, the vertices of the forest are the cut-vertices and the largest $2$-connected subgraphs of the given graph. As will be appreciated in later sections, this transformed structure reveals important details about the connectivity of the given graph.

A \textit{block} of a graph $G$ is a maximal (by inclusion) $2$-connected subgraph of $G$. For index sets $I,J$, let $\{E_i:i\in I\}$ be a partition of $E(G)$ such that each $E_i$ induces a block $B_i$ of $G$, let $\{B_i:i\in J\}$ be the (possibly empty) set of isolated vertices of $G$, and let $\mathrm{BLOCK}(G)=\{B_i:i\in I\cup J\}$. Note that each
non-cut-vertex of $G$ is contained in (or coincides with) exactly one of the $B_i$; each cut-vertex of $G$ is contained in at least two distinct blocks; and for each
$i,j,i\neq j$, $V(B_i)\cap V(B_j)$ consists of at most one vertex, and this vertex (if it exists) is a cut-vertex of $G$. If $B_i$ contains
exactly one cut-vertex of $G$ then $B_i$ is a \textit{leaf block}. An \textit{isolated block} contains no cut-vertices of $G$, i.e., it is a
$2$-connected component of $G$. We use $\mathrm{LEAF}(G)$ to denote the set of leaf blocks of $G$. The \textit{interior} of block
$B_i$ with respect to $G$, denoted $\mathrm{int}(B_i)$, is the set of all vertices of $B_i$ that are not cut-vertices of $G$.

\begin{figure}[htb]
  \begin{center}
    \includegraphics[scale=0.45]{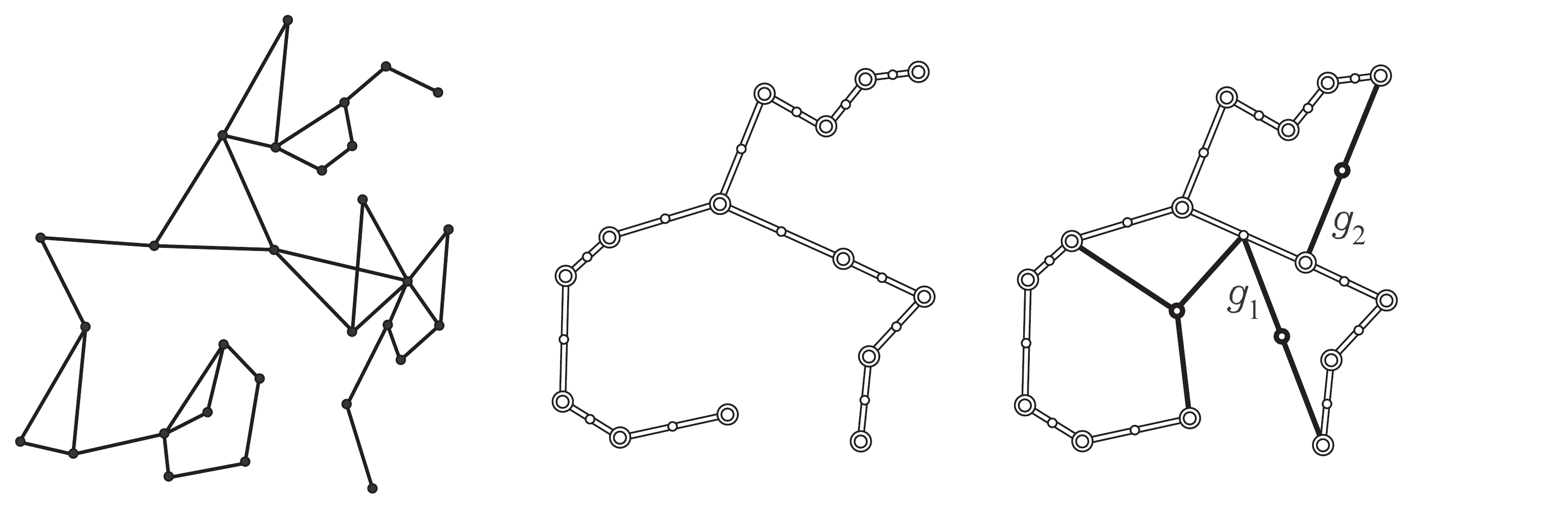}\\
  \end{center}
  \caption{The graphs $G'$, $G_\mathrm{BCF}'$, and $G''$. The $B_i$ are drawn as double-boundary circles}
  \label{figBCF}
\end{figure}

The \textit{block-cut forest} of $G$ is a forest $G_\mathrm{BCF}$ with $V(G_\mathrm{BCF})=\{B_i\in \mathrm{BLOCK}(G)\}\cup
\{z_i:z_i\mathrm{\ is\ a\ cut}\text{-}\mathrm{vertex \ of\ }G\}$ and $E(G_\mathrm{BCF})=\{B_iz_j:z_j\in B_i\}$. In Figure \ref{figBCF} we show three graphs: the first is a graph $G'$ of connectivity $1$; and the second is a depiction of $G_\mathrm{BCF}'$, where large double-boundary circles represent the $B_i$, small circles represent the $z_i$, and edges are drawn as double lines. The purpose of the third figure is to introduce the method by which many subsequent examples will be illustrated in this paper. In this figure we depict a $2$-connected graph $G''$, where $G''$ is the union of a graph $G''_1$ drawn as a block-cut forest and a graph $G_2''$ containing only Steiner edges; as before, Steiner points are represented by small open circles with bold boundaries. When Steiner edges (bold edges) are drawn incident to a double-boundary circle this means that, in $G''$, these edges are incident to vertices in the \textit{interior} of the corresponding block of $G_1''$. Similarly, when a Steiner edge is drawn incident to a small circle this depicts the fact that the Steiner edge is incident to a cut-vertex of $G_1''$.

As the next theorem states, block-cut forests can be constructed in linear time with respect to the number of edges.

\begin{theorem}[see \cite{tarjan}]\label{bcftime}For any $G$ with $m$ edges, $G_\mathrm{BCF}$ can be constructed in time $O(m)$.
\end{theorem}

\section{Underlying networks and candidate types}\label{types}
We begin by describing a process whereby an initial set of terminal edges is added to $X$. We call the graph that results an \textit{underlying network}. As we show, employing underlying networks allows us to reduce the set of all candidate topologies to polynomial size. After dealing with underlying networks we give a detailed description of the division of the set of candidate topologies into \textit{types}, as discussed in Section \ref{over}. We also show how, with respect to any given candidate type, the Steiner points are optimally located so as to achieve minimum bottleneck length.

\subsection{Reducing the number of candidate topologies: underlying networks}\label{under}

Let $S$ be a set of at most $k$ variable Steiner points in the plane, and let $G_\mathrm{UN}=\langle X\cup S,E(G_\mathrm{UN})\rangle$ be any graph such that $E(G_\mathrm{UN})$ contains no Steiner edges. Then $G_\mathrm{UN}$ is referred to as an \textit{underlying network}. As an example we refer back to Figure \ref{figEgs3} where an underlying network results when we remove the dashed edges from the left sub-figure. A \textit{globally (locally) optimal} underlying network is defined as the graph that results by removing the Steiner edges from some globally (locally) optimal network. It is important to note that underlying networks, according to our definition, include a set of isolated Steiner points.

Recall that $n=\vert X\vert$. Since we are only interested in the bottleneck edge when calculating the ``cost" of a network, there are essentially only $O(n^2)$ underlying networks: let $L=\{\ell_1,...,\ell_{n'}\}$ be a non-decreasing set of all distances occurring between pairs of vertices of $X$, where $n'\leq {n\choose 2}$. Let $[K]_i$ be the subgraph of the complete graph on $X$ induced by all edges of length at most $\ell_i$. In the $2$-Bottleneck algorithm, underlying networks are selected from the set $\{[K]_i:i\leq n'\}$. Sets of Steiner edges are then added to the underlying networks to create candidate topologies.

Similarly to a result of Bae et al. \cite{bae1} we have the following result.

\begin{proposition}[\cite{bae1}]\label{specN}There exists a globally optimal network $N^*$ in $L_p$ such the degree of every Steiner point is at most $\Delta=5$ when $1<p<\infty$, and is at most $\Delta=7$ when $p\in\{1,\infty\}$.
\end{proposition}

The previous proposition leads to a further reduction in the number of underlying networks our algorithm needs to consider. For any graph $G$ we use $E_{\text{SE}}(G)$ to denote the set of Steiner edges of $G$.

\noindent\textbf{Definition.} \textit{If $N:=G_\mathrm{UN}+E_{\mathrm{SE}}(N)$ is $2$-connected then $N$ is referred to as an \textbf{augmented network} containing $G_\mathrm{UN}$.}

\begin{lemma}[\cite{brazil2}]\label{lemLeaf2}Let $N$ be an augmented network containing $G_\mathrm{UN}$. For every leaf-block $B$ of $G_\mathrm{UN}$ there exists at least one Steiner edge in $N$ incident to a vertex in the interior of $B$. For every isolated block $W$ of $G_\mathrm{UN}$ there exists at least two distinct Steiner edges in $N$ incident to vertices in $W$.
\end{lemma}

For any $G$, let $b(G)$ be the number of leaf blocks plus twice the number of isolated blocks occurring in $G$ (recall that isolated vertices and edges are blocks according to definition).

\begin{corollary}\label{leafLim}There exists a globally optimal underlying network $G_\mathrm{UN}^*$ such that $b(G_\mathrm{UN}^*)\leq \Delta k$.
\end{corollary}
\begin{proof}
Let $N^*$ be a globally optimal network satisfying Proposition \ref{specN} and let $G_\mathrm{UN}^*$ be the graph that results by removing all Steiner edges from $N^*$ and keeping all vertices. Then $G_\mathrm{UN}^*$ is an underlying network, $N^*$ is an augmented network containing $G_\mathrm{UN}^*$, and the corollary follows from Lemma \ref{lemLeaf2}.
\end{proof}

A naive algorithm for constructing optimal bottleneck $2$-connected $k$-Steiner networks would therefore need to consider roughly $O(n^{\Delta k+2})$ topologies: $O(n^2)$ underlying networks multiplied by $O(n^{\Delta k})$, since there are at most $\binom{n+k}{\Delta}$ neighbour-set choices for each Steiner point. Our $2$-Bottleneck algorithm will reduce this complexity even further by means of candidate types and farthest-colour Voronoi diagrams.

\subsection{Candidate types and Steiner endpoint sequences}\label{augG}
Let $\mathcal{Q}=\langle (s_1,Y_1),...,(s_\rho,Y_\rho) \rangle$, where $\rho\leq \Delta k$, be an arbitrary sequence such that $s_i\in S$ and either $Y_i$ is the singleton $\{s_i'\}$, where $s_i'\in S$, or $Y_i\subseteq X$. Intuitively, the pair $(s_i,Y_i)$ contains the potential endpoints of a Steiner edge in a candidate topology. For instance, in the example from Figure \ref{figEgs3} at the end of Section \ref{over}, the following sequence satisfies this form: $\mathcal{Q}_0=\langle(\{s\},\{x,x_1\}),(\{s\},\{y\}), (\{s\},Z)\rangle$, where $s$ is the Steiner point, $x,x_1$ are the two red vertices, $y$ is the blue vertex, and $Z$ is the set of all grey vertices. Note therefore that for any $i,j$, the Steiner points $s_i$ and $s_j$ (or $s_i'$ and $s_j'$) are not necessarily distinct.

A sequence $\mathcal{Q}$ and an underlying network together define a candidate type as follows. For arbitrary $\mathcal{Q}$, let $G(\mathcal{Q}):=G_\mathrm{UN}+E_{\mathrm{SE}}(G(\mathcal{Q}))$, where $E_{\mathrm{SE}}(G(\mathcal{Q}))=\{h_1,...,h_\rho\}$ is a set of \textit{labelled} Steiner edges with variable endpoints such that, for every $1\leq j\leq \rho$, one endpoint of $h_j$ is $s_j$ and the other is in $Y_j$. We call $G(\mathcal{Q})$ a \textit{candidate type}. The sequence $\mathcal{Q}$ is called a \textit{Steiner endpoint sequence} for $G(\mathcal{Q})$. Let $G_{\text{OPT}}(\mathcal{Q})$ be a cheapest network derived from $G(\mathcal{Q})$ by optimally locating the Steiner edges (as restricted by $\mathcal{Q}$) and Steiner points in the plane. Note that for arbitrary $\mathcal{Q}$, the graph $G_{\mathrm{OPT}}(\mathcal{Q})$ is not necessarily $2$-connected. In the example of Figure \ref{figEgs3} the graph $G_{\mathrm{OPT}}(\mathcal{Q}_0)$ is illustrated in the right sub-figure. We provide another example of the above concepts in Figure \ref{figEgs22}.

\begin{figure}[htb]
  \begin{center}
    \subfigure[$G_\mathrm{UN}$]{\label{figEgs22a}\includegraphics[scale=0.45]{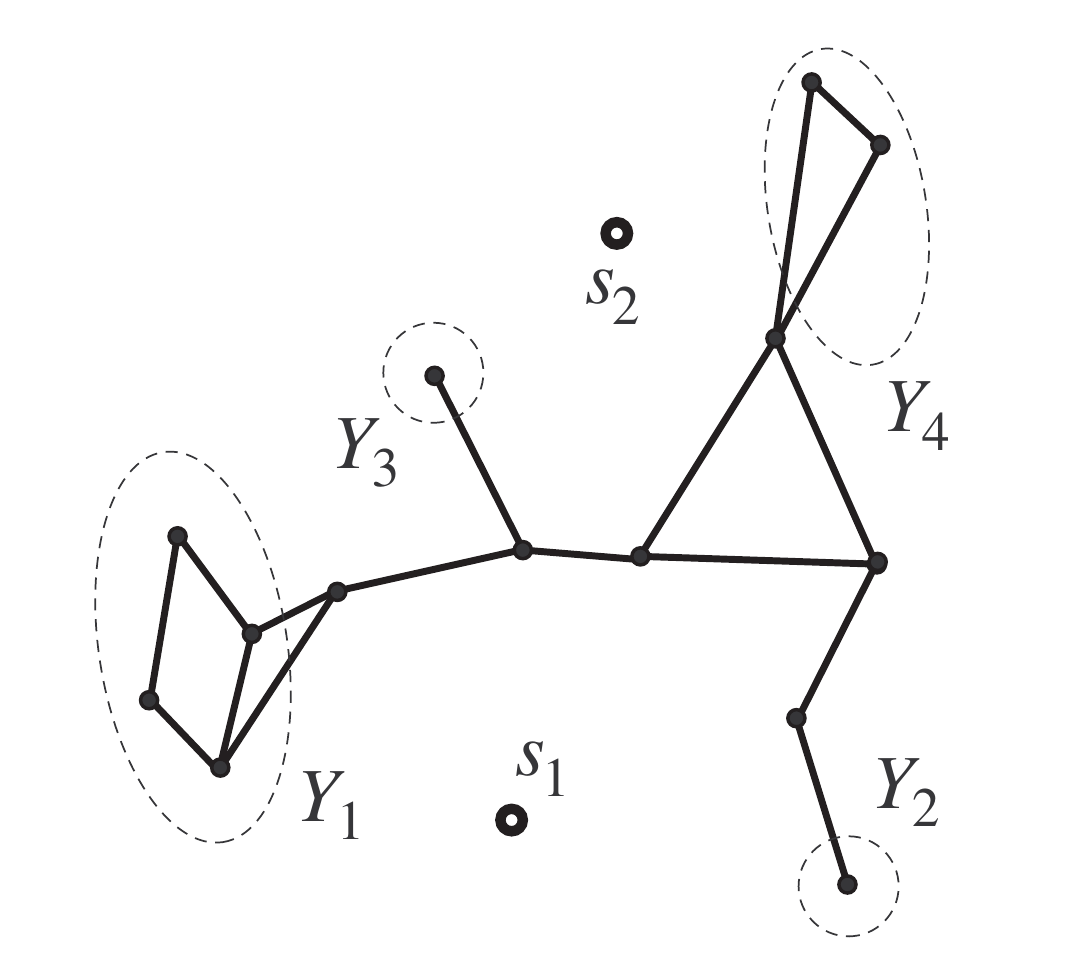}}
    \subfigure[$G(\mathcal{Q}_1)$]{\label{figEgs22b}\includegraphics[scale=0.45]{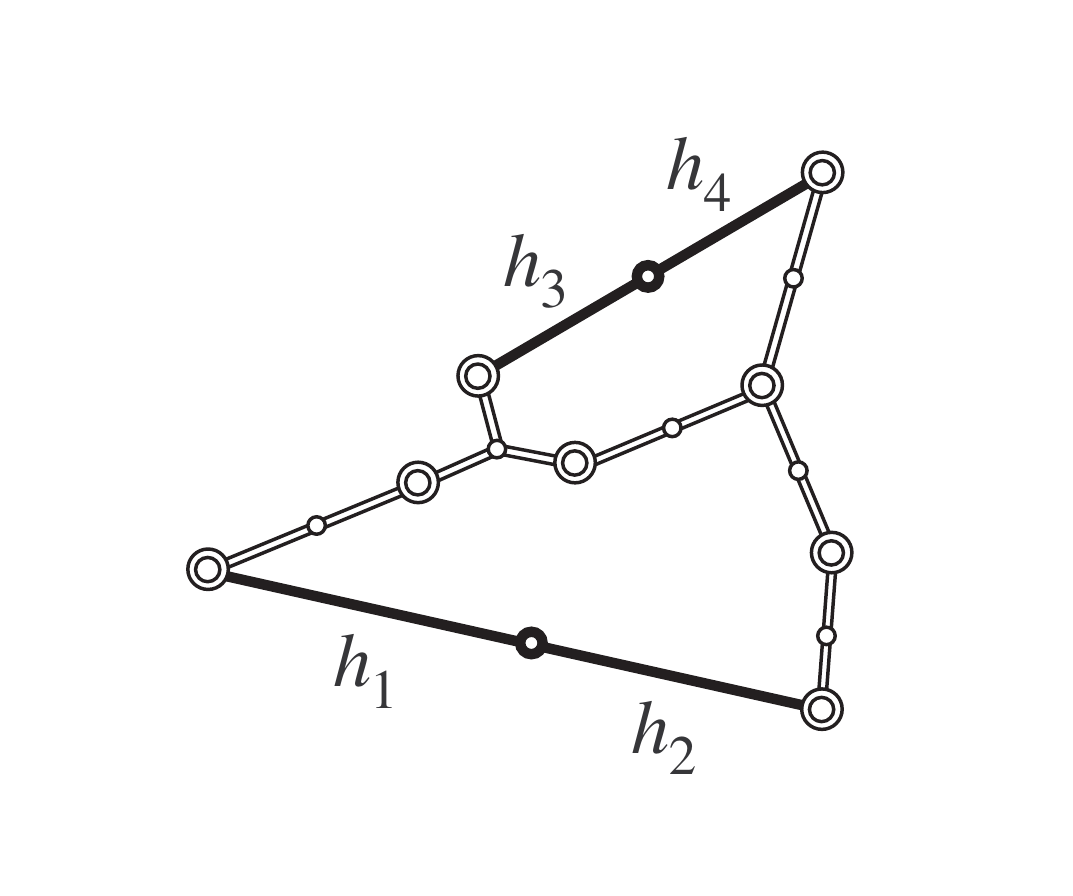}}
    \subfigure[$G_{\mathrm{OPT}}(\mathcal{Q}_1)$]{\label{figEgs22c}\includegraphics[scale=0.45]{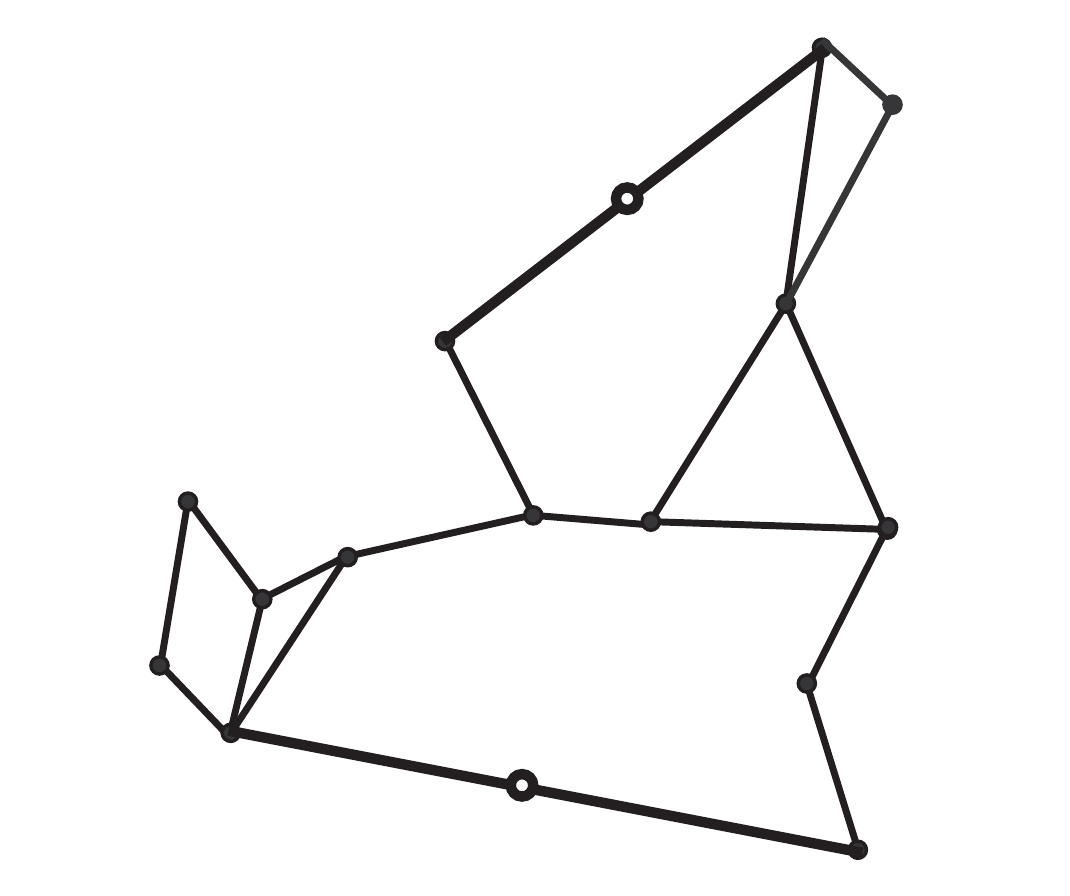}}
  \end{center}
  \caption{An example of an underlying network $G_\mathrm{UN}$, a candidate type $G(\mathcal{Q}_1)$ containing the underlying network, and the resultant graph $G_{\mathrm{OPT}}(\mathcal{Q}_1)$ when $\rho=4$ and $\mathcal{Q}_1=\langle(s_1,Y_1),(s_1,Y_2),(s_2,Y_3),(s_2,Y_4)\rangle$. The candidate type $G(\mathcal{Q}_1)$ is illustrated by drawing $G_\mathrm{UN}-S$ as a block-cut forest}
  \label{figEgs22}
\end{figure}

Let $\mathcal{Q}=\langle (s_1,Y_1),...,(s_\rho,Y_\rho) \rangle$ and $\mathcal{Q}'=\langle (s_1,Y_1'),...,(s_\rho,Y_\rho') \rangle$ be two Steiner endpoint sequences. Recall that $\ell_\text{max}(G_{\text{OPT}}(\mathcal{Q}))$ is the length of a longest edge of $G_{\text{OPT}}(\mathcal{Q})$. The definition of $G_{\mathrm{OPT}}(\mathcal{Q})$ implies the following property.

\begin{proposition}[Monotonicity property]If $Y_i\subseteq Y_i'$ for all $i$ then\newline $\ell_\mathrm{max}(G_{\mathrm{OPT}}(\mathcal{Q}'))\leq\ell_\mathrm{max}(G_{\mathrm{OPT}}(\mathcal{Q}))$.
\end{proposition}

Candidate types, as formally defined above, are useful in the following way. Let $N^*$ be a globally optimal network with Steiner edges $s_1y_1,...,s_\rho y_\rho$, where the $s_i$ are not necessarily distinct. Let $\mathcal{Q}^*=\langle (s_1,Y_1^*),...,(s_\rho,Y_\rho^*) \rangle$ be any Steiner endpoint sequence such that $y_i\in Y_i^*$ for every $i$, and such that $G_{\text{OPT}}(\mathcal{Q}^*)$ is $2$-connected (where $G_{\text{OPT}}(\mathcal{Q}^*)$ contains the same underlying network as $N^*$). There exists at least one such sequence, since $N^*=G_{\text{OPT}}(\mathcal{Q}^*_0)$ where $\mathcal{Q}^*_0=\langle (s_1,\{y_1\}),...,(s_\rho,\{y_\rho\}) \rangle$. Since $N^*$ is optimal, we have $\ell_\mathrm{max}(N^*)\leq \ell_\mathrm{max}(G_{\text{OPT}}(\mathcal{Q}^*))$. By the Monotonicity property, we also have $\ell_\mathrm{max}(N^*)\geq \ell_\mathrm{max}(G_{\text{OPT}}(\mathcal{Q}^*))$. Hence $G_{\text{OPT}}(\mathcal{Q}^*)$ is a globally optimal network. Therefore we immediately obtain a fast (polynomial-time) algorithm for constructing a globally optimal network if there exists a set of Steiner endpoint sequences $\Lambda$ with the following properties.

\begin{enumerate}
    \item[A.] The class $\Lambda$ is small and has a fast explicit construction,
    \item[B.] There exists a globally optimal network $N^*$ with Steiner edges $s_1y_1,...,s_\rho y_\rho$, and a Steiner endpoint sequence $\mathcal{Q}^*\in \Lambda$, such that $y_i\in Y_i^*$ for all $i$,
    \item[C.] For every $\mathcal{Q}\in \Lambda$ there exists a fast explicit construction of $G_{\text{OPT}}(\mathcal{Q})$,
    \item[D.] For every $\mathcal{Q}\in \Lambda$ the graph $G_{\text{OPT}}(\mathcal{Q})$ is 2-connected.
\end{enumerate}

In particular, an algorithm utilising these properties simply constructs $\Lambda$ and then selects a cheapest network that results from constructing $G_{\text{OPT}}(\mathcal{Q})$ for every $\mathcal{Q}\in \Lambda$. The bulk of this paper is devoted to describing how a set $\Lambda$ satisfying Properties (A)--(D) is constructed.

In the next section we demonstrate that Property (C) holds for a broad class of Steiner endpoint sequences. Section \ref{nolink} constructs a set of Steiner endpoint sequences $\Lambda_0$ then goes on to prove that Properties (A) and (B) are satisfied for the pair $\mathcal{N}_0,\Lambda_0$, where $\mathcal{N}_0$ is a certain restricted class of augmented networks. In Section \ref{makeCon} we transform the set $\Lambda_0$ into a new set $\Lambda_1$ such that Properties (A)--(D) hold for the pair $\mathcal{N}_0,\Lambda_1$. Finally, in Section \ref{withLink} we extend $\Lambda_1$ to a set $\Lambda$ so that all four properties hold for the pair $\mathcal{N},\Lambda$, where $\mathcal{N}$ is the set of all augmented networks containing a given underlying network.

\subsubsection{Constructing $G_{\text{OPT}}(\mathcal{Q})$}\label{jusdoit}
Motivated by Corollary \ref{leafLim} we assume for the remainder of this paper, until we present the $2$-Bottleneck algorithm, that $G_\mathrm{UN}$ denotes a fixed but arbitrary underlying network with $b(G_\mathrm{UN})\leq \Delta k$. Unless stated otherwise, any augmented network is assumed to contain $G_\mathrm{UN}$. For any Steiner endpoint sequence $\mathcal{Q}=\langle (s_1,Y_1),...,(s_\rho,Y_\rho) \rangle$, a \textit{representative of} the candidate type $G(\mathcal{Q})$ is any graph $G=G_\mathrm{UN}+ E_\mathrm{SE}(G)$, where $E_\mathrm{SE}(G)=\{s_1y_1,...,s_\rho y_\rho\}$ and $y_i\in Y_i$ for every $i$.
Clearly then $G_{\text{OPT}}(\mathcal{Q})$ is a representative of $G(\mathcal{Q})$. The \textit{Steiner topology} of a graph $G$ is the unlabelled topology of the graph induced by the Steiner edges of $G$. Since each $Y_i$ is a subset of $X$ or is a singleton containing a Steiner point it follows that the Steiner topologies of any pair of representatives of $G(\mathcal{Q})$ are isomorphic; hence we refer to this topology as the \textit{Steiner topology of} $G(\mathcal{Q})$.

The method of Bae et al. \cite{bae1}, which utilises farthest colour Voronoi diagrams to construct $G_\mathrm{OPT}(\mathcal{Q})$, only applies to trees. However, we demonstrate in this subsection that there exists a globally optimal network with an acyclic Steiner topology. This allows the use of Bae et al.'s method for the $2$-connected case -- a fact which is formalised in the next theorem.

\begin{theorem}\label{consOp}Let $\mathcal{Q}$ be a Steiner endpoint sequence such that the Steiner topology of $G(\mathcal{Q})$ is acyclic. If the degree of every Steiner point in $G(\mathcal{Q})$ is bounded by $\Delta$ then $G_{\mathrm{OPT}}(\mathcal{Q})$ can be constructed in time $O(n^k)$ in $L_p$ for $1<p<\infty$, and in time $O(n\log^2n)$ time for $L_1$ and $L_\infty$, where $k$ is the number of Steiner points.
\end{theorem}
\begin{proof}
The result follows directly from \cite{bae1}: the ``abstract topology" concept in \cite{bae1} is analogous to our concept of candidate types.
\end{proof}

To demonstrate that there exists a globally optimal network with an acyclic Steiner topology we begin with the following definitions. Let $G$ be a $2$-connected graph. The removal of a \textit{critical edge} reduces the connectivity of $G$. A \textit{chord path} in $G$ is a path $P$ connecting two points of a cycle $C$ of $G$, such that $P$ and $C$ share no edges. An edge $e$ is critical in $G$ if and only if $e$ is not a chord path \cite{dirac}. We define a \textit{degree-two chord path} as a chord-path where all interior vertices are degree-two Steiner points. A chord path $P$ of $G$ is \textit{critical} if the removal of $P$ from $G$ reduces connectivity. Therefore degree-two chord paths are never critical. Let $N$ be any augmented network.

\begin{lemma}\label{critStuff}If all Steiner edges and chord paths of $N$ are critical then there is no chord path in $N$ where all interior vertices are Steiner points.
\end{lemma}
\begin{proof}
Of all chord paths of $N$ consisting entirely of Steiner points let $P$ be one containing the least number of edges, and let $C$ be a cycle in $N$ of which $P$ is a chord path. Let the end-vertices of $P$ be $v,w$, let $s$ be a Steiner point of degree at least three in the
interior of $P$, and let $u$ be a neighbour of $s$ not on $P$. Finally, let $P_v,P_w$ be the subpaths of $P$ that partition the edge-set of $P$
at $s$. By the $2$-connectivity of $N$ there exists a path $P^\prime$ that connects $u$ and $v$ but does not contain $s$. Regardless of the
location of the first intersection of $P^\prime$ with $C\cup P$ a new cycle is produced with a chord path formed by a subpath of $P_v$ or $P_w$.
This contradicts the fact that $P$ is a chord path with the least number of edges.
\end{proof}

\begin{theorem}\label{Tcord}If all Steiner edges and chord paths of $N$ are critical, then there is no cycle in $N$ consisting of Steiner points only.
\end{theorem}
\begin{proof}
Any cycle of a $2$-connected graph $N$, where $N$ is not a cycle, contains a subpath that is a chord of some cycle. Therefore the result follows from the previous lemma.
\end{proof}

If we consecutively remove from any augmented graph all non-critical edges and degree-two chord paths in an arbitrary order, the resultant graph remains $2$-connected and the length of its bottleneck does not increase. Therefore, by Theorem \ref{Tcord}, there exists a globally optimal network with an acyclic Steiner topology. This result, together with Proposition \ref{specN}, means we only need to consider Steiner endpoint sequences satisfying the two conditions of Theorem \ref{consOp}. Therefore we do not explicitly address Property (C) again in this paper. Observe also that these two conditions can be verified in constant time for a given Steiner endpoint sequence. The next subsection initiates the description of $\Lambda$.

\subsubsection{Augmented networks without linked sets}\label{nolink}
In specifying $\Lambda$ we first restrict our attention to augmented graphs that do not contain \textit{linked sets}, as defined below. When considering the general case (where linked sets are included) in Section \ref{withLink} and Section \ref{final}, the main subroutine described in this subsection, namely Function \texttt{BuildSES}, will be employed as part of a pre-processing stage. We first define the concept of linked sets and then state a number of results and definitions before presenting Function \texttt{BuildSES}.

A \textit{block path} of a graph $G$ is a subgraph $H$ such that $H_\mathrm{BCF}$ is a path in $G_\mathrm{BCF}$. If the blocks of a block path $H$ are written in order of adjacency as $B_1,...,B_{p_H}$ then there are two possible \textit{orientations} of $H$ which satisfy this order. A \textit{degree-two block path} of $G$ is a \textit{maximal} length block path $H$ such that the interior vertices of $H_\mathrm{BCF}$ are of degree two in $G_\mathrm{BCF}$. An \textit{external} Steiner edge is one that is incident to a terminal; all other Steiner edges are \textit{internal}.

\begin{figure}[htb]
  \begin{center}
    \includegraphics[scale=0.45]{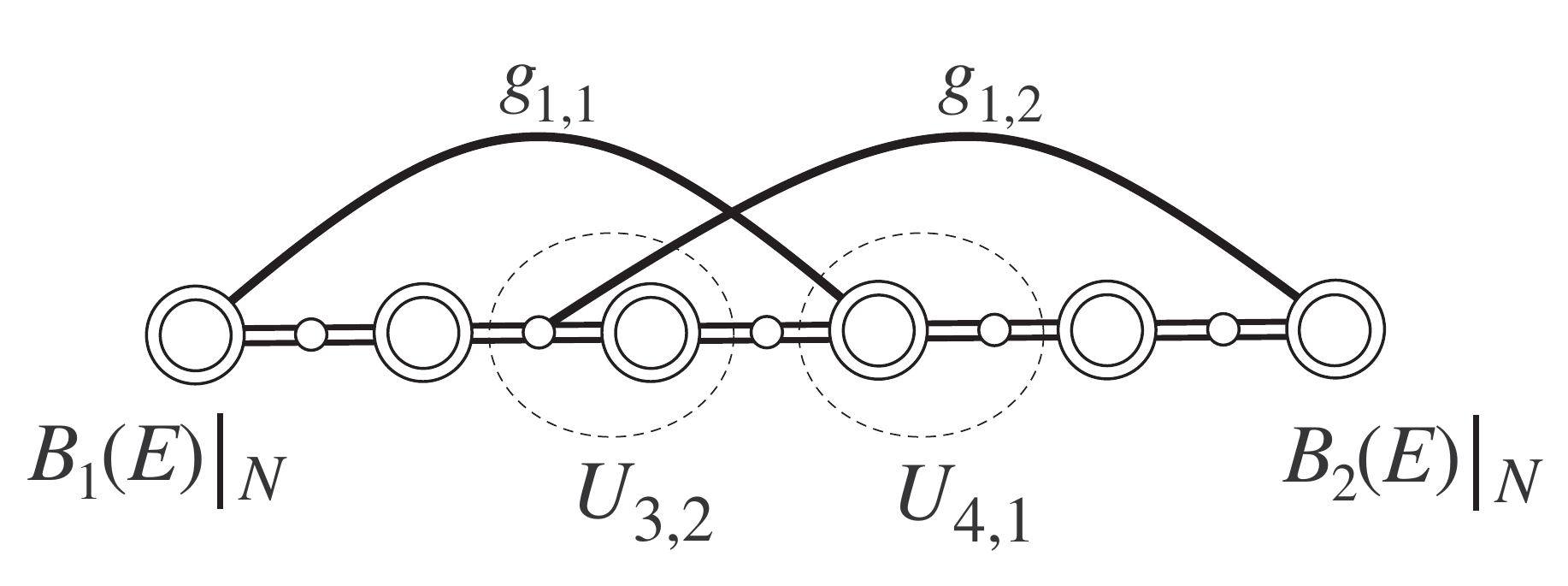}
  \end{center}
  \caption{A linked set $E_{\mathrm{LINK}}=\{g_{1,1},g_{1,2}\}$ with $a_1=4$ and $a_1'=3$}
  \label{figEgNew2}
\end{figure}

Let $N$ be an augmented network such that all Steiner edges and chord paths are critical. A \textit{path-forming} set of edges in $N$ is any set $E$ of Steiner edges of $N$ such that $N-E$ is a block path. Since all Steiner edges of $N$ are critical, any singleton Steiner edge set is path-forming. We denote the distinct leaf-blocks of $N-E$ by $B_1(E)\vert_N$ and $B_2(E)\vert_N$ (which are unique up to the orientation of $N-E$) and slightly abuse this notation by writing $B_1(e)\vert_N$ when we mean $B_1(\{e\})\vert_N$.

\noindent\textbf{Definition.} \textit{A path-forming set $E_{\mathrm{LINK}}=\{g_1,g_2\}$ is called a \textbf{linked set} in $N$ if $g_{1},g_{2}$ are external Steiner edges with Steiner endpoints in $\mathrm{int}(B_1(E_{\mathrm{LINK}})\vert_N)$ and $\mathrm{int}(B_2(E_{\mathrm{LINK}})\vert_N)$ respectively for some orientation.}

See Figure \ref{figEgNew2} for an illustration of a linked set $\{g_{1,1},g_{1,2}\}$ . The set $\{g_1,g_2\}$ in the graph $G''$ of Figure \ref{figBCF} is another example of a linked set.

Let $\mathcal{N}$ be the class of all augmented networks (containing $G_\mathrm{UN}$), such that every $N\in\mathcal{N}$ contains at most $\Delta k$ Steiner edges, and such that the chord paths and Steiner edges of every member of $\mathcal{N}$ are critical. Let $\mathcal{N}_0\subset \mathcal{N}$ be the class of all graphs in $\mathcal{N}$ that have no linked sets.

\begin{lemma}[\cite{brazil2}]\label{joinB}If $G$ is a block path with leaf-blocks $B,B'$, then, for any $x\in \mathrm{int}(B)$ and $x'\in \mathrm{int}(B')$, the graph $G+xx'$ is $2$-connected.
\end{lemma}

By the above lemma, the class $\mathcal{N}_0$ is nonempty since any linked set $\{s_1x,s_2x'\}$ in some $N'\in\mathcal{N}$ can be replaced by the single edge $s_1s_2$, resulting in a $2$-connected graph with fewer linked sets than $N'$.

Function \texttt{BuildSES}, presented below, essentially constructs a set of Steiner endpoint sequences $\Lambda_0$ satisfying properties analogous to Properties (A)--(D), except that the representatives of any candidate topology with a sequence in $\Lambda_0$ belong to $\mathcal{N}_0$. We first need a few more preliminary definitions and results.

Let $P=\langle e_1,...,e_{\vert E_\mathrm{SE}(N)\vert}\rangle$ be an ordering of the Steiner edges of some $N\in{\mathcal{N}}_0$. Let $N_j:=N-\{e_i:1\leq i\leq j\}$ for $1\leq j\leq p_N$. Furthermore, suppose that the following conditions hold on $P$:
\begin{enumerate}
\item If $N_j$ is connected, where $j\geq 1$, then $e_j=xy$ where $x\in \mathrm{int}(B),y\in \mathrm{int} (B')$ for some pair of distinct leaf blocks $B,B'$ of $N_j$.
\item If $N_j$ is disconnected, where $j\geq 2$, then $e_j=xy$ where $x$ and $y$ are in distinct components of $N_j$.
\end{enumerate}

We then call $P$ a \textit{branching decomposition} of $N$.

\begin{figure}[htb]
  \begin{center}
    \subfigure[]{\label{figEgs4-1}\includegraphics[scale=0.42]{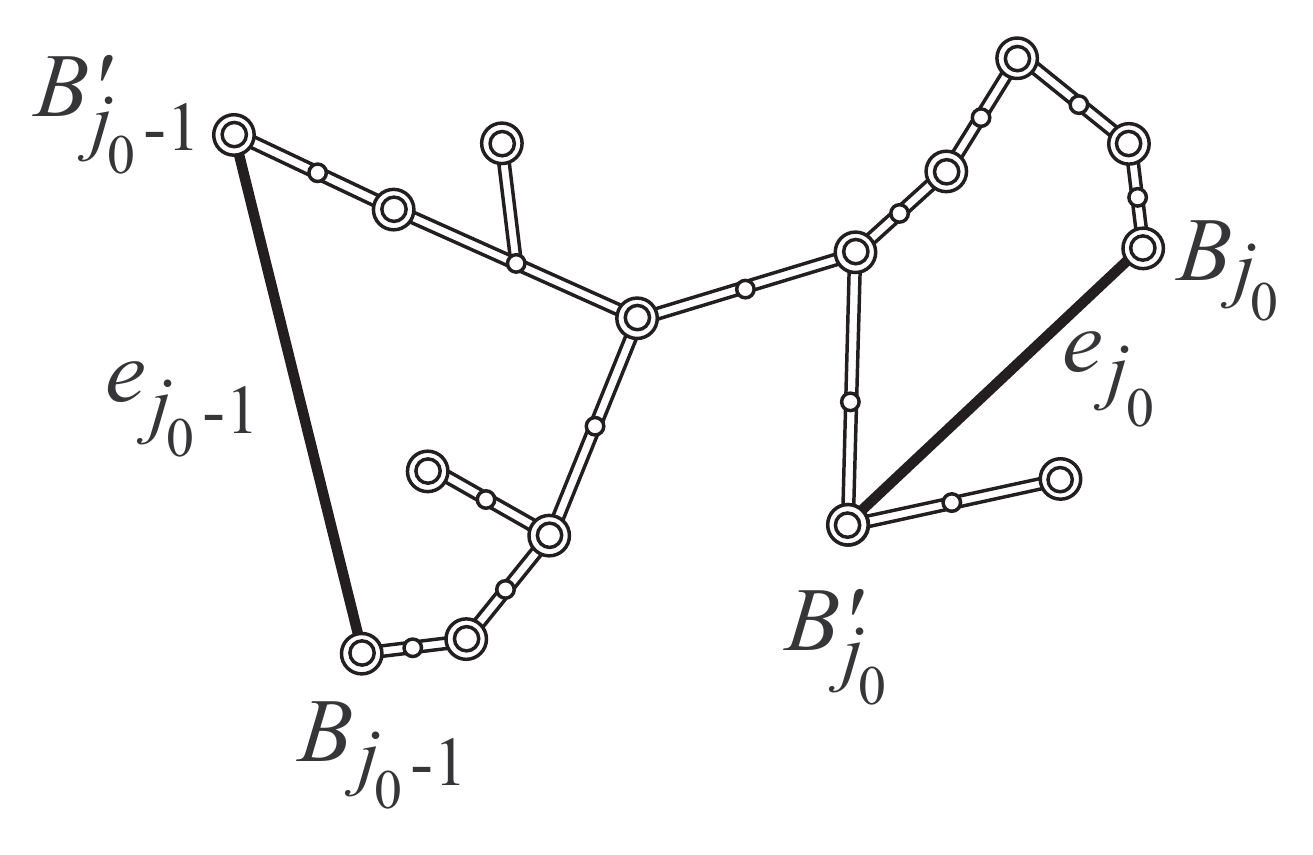}}
    \subfigure[]{\label{figEgs4-2}\includegraphics[scale=0.42]{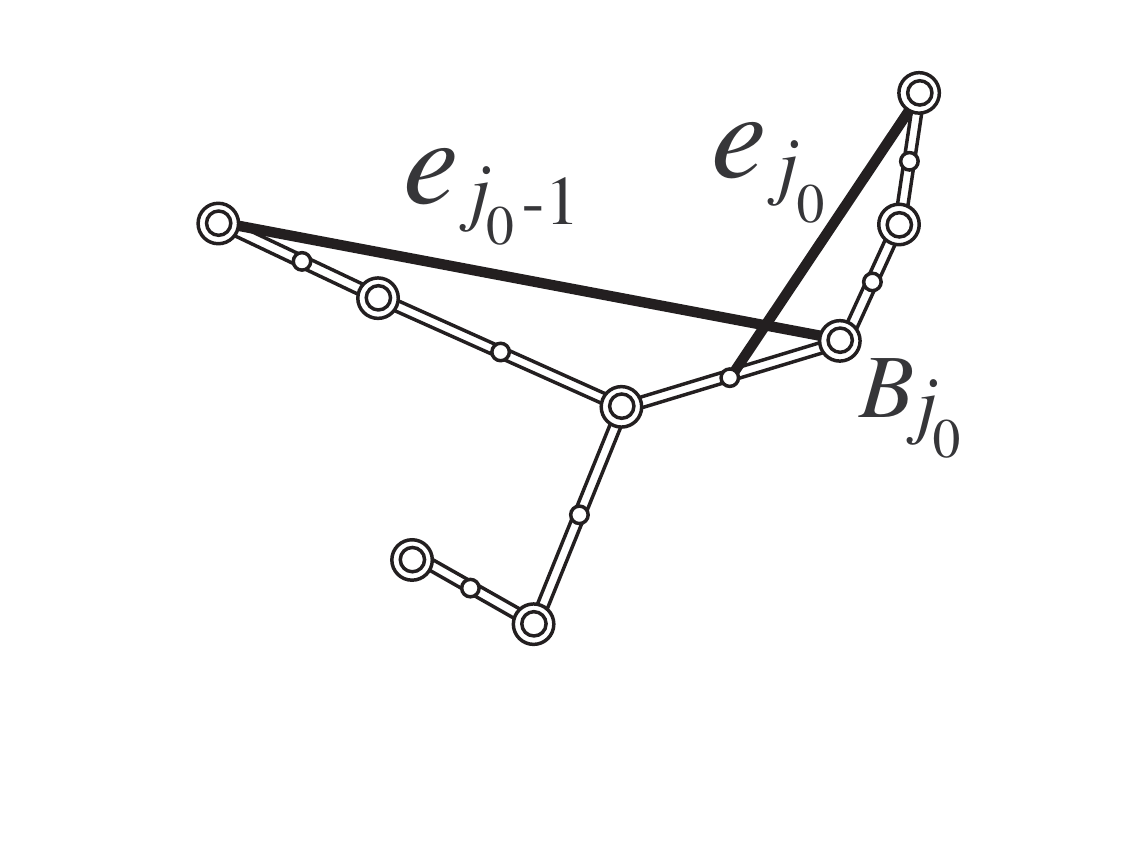}}
  \end{center}
  \caption{Examples of two cases from Proposition \ref{anyBranch}}
  \label{figEg4}
\end{figure}

\begin{proposition}\label{anyBranch}Any $N\in{\mathcal{N}}_0$ has a branching decomposition.\label{hasBranch}
\end{proposition}
\begin{proof}
Let $N$ be any graph in ${\mathcal{N}}_0$ and let $E_0$ be a maximal set of Steiner edges that can be removed from $N$ so that $N-E_0$ is connected. Suppose that we remove the edges in $E_0$ from $N$ in some order $P_0=\langle e_{1},...,e_{\vert E_0\vert}\rangle$, and let $N_j$ be defined as before on $0\leq j\leq \vert E_0\vert$, with $N_0:= N$. If, for some $j\geq 1$, it holds that $e_{j}$ is incident, in ${N}_{j-1}$, to the interiors of two distinct leaf blocks of ${N}_{j}$ we say that $j$ is a \textit{branching index}. Clearly, since all Steiner edges of $N$ are critical, $j=1$ is a branching index. We claim that there exists an ordering $P_0$ of $E_0$ such that every $1\leq j\leq \vert E_0\vert$ is a branching index.

Suppose that the claim is false. Let $P'$ be an ordering of $\{e_{1},...,e_{\vert E_0\vert}\}$ that minimises the smallest non-branching index, and let $j_0>1$ be the smallest non-branching index in $P'$. For each $j\in\{j_0-1,j_0\}$, let $\mathcal{P}_{j}$ be the shortest block path of ${N}_{j_0}$ such that the leaf blocks of $\mathcal{P}_{j}$, say $B_{j},B_{j}'$, contain the respective endpoints of $e_j$ in $N$; see, for instance, Figure \ref{figEg4}. Therefore, in fact, the endpoints of $e_j$ are in the interiors (with respect to $\mathcal{P}_{j}$) of $B_{j}$ and $B_{j}'$. Hence, $\mathcal{P}_{j}+e_j$ is a block of ${N}_{j_0}+e_{j}$. Note that $\mathcal{P}_{j_0-1}+e_{j_0-1}$ is the only block of ${N}_{j_0}+e_{j_0-1}$ which is not a block of ${N}_{j_0}$.

We have three cases:

\begin{enumerate}
    \item At least one of $B_{j_0}, B_{j_0}'$ is neither a leaf-block of ${N}_{j_0}$, nor a block of $\mathcal{P}_{j_0-1}$. Then, since $j_0-1$ is a branching index and $\mathcal{P}_{j_0}+e_{j_0}$ is not a leaf block of ${N}_{j_0}+e_{j_0}$, both leaf blocks of $\mathcal{P}_{j_0-1}$ are leaf-blocks of ${N}_{j_0}$; see Figure \ref{figEgs4-1}. Hence we may swap the edges $e_{j_0}$ and $e_{j_0-1}$ in the ordering $P'$ to produce a new ordering which has a smallest non-branching index of $j_0-1$. This contradicts the minimality of $j_0$.
    \item $B_{j_0}$ and $B_{j_0}'$ are both in $\mathcal{P}_{j_0-1}$. But then $e_{j_0}$ is a chord of a cycle in $N$, and is therefore not critical. This contradicts the assumption that $N\in\mathcal{N}_0$.
    \item $B_{j_0}$ is a leaf-block of ${N}_{j_0}$, and $B_{j_0}'$ is in $\mathcal{P}_{j_0-1}$ (or vice-versa). This gives two subcases.
        \begin{enumerate}
            \item $\mathcal{P}_{j_0}+e_{j_0}$ is a leaf-block of ${N}_{j_0}+e_{j_0}$, and some endpoint of $e_{j_0-1}$, say $x$, is not in the interior of a leaf-block of ${N}_{j_0}$; see Figure \ref{figEgs4-2}. Then, since $j_0-1$ is a branching index, $x$ must be in the interior of $\mathcal{P}_{j_0}+e_{j_0}$ (with respect to ${N}_{j_0}+e_{j_0}$). But then $\{e_{j_0},e_{j_0-1}\}$ is a linked set in $N$,
                which is a contradiction.
            \item Both endpoints of $e_{j_0}$ are contained in leaf-blocks of ${N}_{j_0}$. By switching $e_{j_0}$ and $e_{j_0-1}$, as before, we get a new ordering where the index $j_0$ is a branching index. This either means that there are no non-branching indices in the new ordering (which happens when $\mathcal{P}_{j_0-1}+e_{j_0-1}$ is a leaf-block), or the new ordering contains a non-branching index at $j_0-1$. In both cases we get a contradiction.
        \end{enumerate}
\end{enumerate}

In all three above cases we get a contradiction. Therefore no such $P'$ exists, which means some ordering, say $P_0$, exists which has no non-branching indices. Next let $\overline{E}_0=E_{\mathrm{SE}}(N)-E_0$. Since $E_0$ is maximal every Steiner edge in $\overline{E}_0$ is a bridge of $N-E_0$. Clearly, for any graph $G$, if $e,e'$ are bridges of $G$ then $e'$ is a bridge in a component of $G-e$. Therefore, removing the edges in $\overline{E}_0$ in any order, say $\overline{P}_0$ from $N-E_0$ will increase the number of components of the resulting graph at each step. Therefore $P=(P_0,\overline{P}_0)$ is a branching decomposition of $N$.
\end{proof}

Function \texttt{BuildSES} populates the afore-mentioned set $\Lambda_0$ by constructing augmented networks in an iterative manner, where the pairs comprising the corresponding Steiner endpoint sequences are selected at each step. Starting with $G_\mathrm{UN}$, the function adds edges in the order of a reverse branching decomposition. At each addition of an edge $f_{t+1}$ to the current graph $F_t$, a pair $(s,Y)$ is chosen as the $(t+1)$-th pair in a Steiner endpoint sequence. Therefore $Y$ is either a singleton containing a Steiner point, or is a subset of terminals. The pairs $(s,Y)$ selected by the function are referred to as \textit{valid pairs}. When some choice of $Y$ at an iteration of Function \texttt{BuildSES} contains only terminals we refer to $Y$ as a \textit{valid subset of} $X$. We next define valid pairs and valid subsets more rigourously. We define the valid subsets so that they create a partition of $X$ with $O(1)$ members. This will ensure that Function \texttt{BuildSES} runs in polynomial time and that every network in $\mathcal{N}_0$ is a representative of a candidate type of some sequence in $\Lambda_0$. This latter property will be demonstrated in Theorem \ref{mainN}.

\begin{figure}[htb]
  \begin{center}
    \includegraphics[scale=0.6]{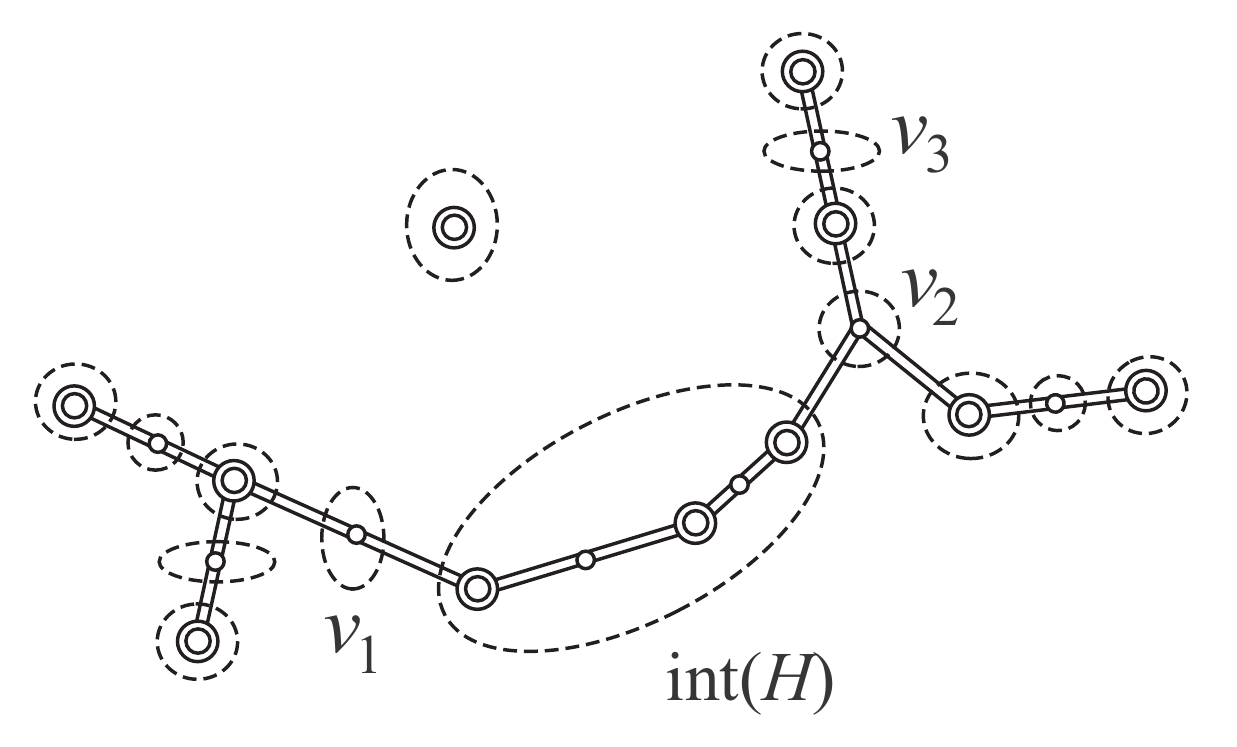}\\
  \end{center}
  \caption{Example of a partition of $X$ into valid subsets}
  \label{figEgs9}
\end{figure}

Let $V_\mathrm{CV}$ be the set of all cut-vertices of $G_\mathrm{UN}$ such that every vertex in $V_\mathrm{CV}$ is either of degree at least three in $(G_\mathrm{UN})_\mathrm{BCF}$, or is contained in a block of $G_\mathrm{UN}$ that is not of degree two in $(G_\mathrm{UN})_\mathrm{BCF}$; see Figure \ref{figEgs9} where $v_1,v_2$ and $v_3$ (amongst others) belong to $V_\mathrm{CV}$. Let $H$ be a degree two block path of $G_\mathrm{UN}$. The \textit{interior} of $H$ with respect to $G_\mathrm{UN}$ is the graph $\mathrm{int}(H)=H-\mathcal{B}-V_\mathrm{CV}$, where $\mathcal{B}$ contains all blocks of $H$ that are not of degree two in $(G_\mathrm{UN})_\mathrm{BCF}$. Let $\mathcal{V}$ be a partition of $X$ such that each member $X'$ of $\mathcal{V}$ is one of the following: the interior of a leaf-block of $G_\mathrm{UN}$; the interior of a degree-two block path of $G_\mathrm{UN}$; the vertices of an isolated block of $G_\mathrm{UN}$; or a singleton containing a vertex of $V_\mathrm{CV}$ (see Figure \ref{figEgs9}). Clearly $\mathcal{V}$ is unique and contains $O(1)$ members by the assumption (in Section \ref{jusdoit}) that $b(G_\mathrm{UN})\leq \Delta k$. The members of $\mathcal{V}$ are referred to as \textit{valid subsets of} $X$.

For any graph $F_t=G_\mathrm{UN}+ E_\mathrm{SE}(F_t)$, a pair $(s,Y)$, where $s\in S$ and $Y\subseteq X\cup S$, is called a \textit{valid pair} for $F_t$ if the following conditions hold.

\makeatletter
\renewcommand{\theenumi}{\Roman{enumi}}
\renewcommand{\labelenumi}{\theenumi.}
\renewcommand{\theenumii}{\alph{enumii}}
\renewcommand{\labelenumii}{\theenumii.}
\renewcommand{\p@enumii}{\theenumi.}
\makeatother
\begin{enumerate}
\item $Y$ is either a singleton containing a Steiner point $s'$ with $s\neq s'$, or $Y$ contains only terminals. Furthermore, there exists $u\in Y$ so that $su$ is not an edge of $F_t$.\label{q1}
\item If $F_t$ is disconnected then:
    \begin{enumerate}
        \item $s$ and $Y$ are in distinct components of $F_t$, and\label{q21}
        \item If $Y$ contains terminals then $Y$ is a valid subset of $X$.\label{q22}
    \end{enumerate}
\item  If $F_t$ is connected then:
    \begin{enumerate}
        \item $s$ is in the interior of a leaf-block $B$ of $F_t$, and\label{q31}
        \item If $Y=\{s'\}$ then $s'$ is in the interior of a leaf-block of $F_t$, distinct from $B$, and\label{q32}
        \item If $Y\subseteq X$ then $Y$ is a maximal subset of $X'\cap V(B')$, where $X'$ is some valid subset of $X$, and $B'$ is a leaf block of $F_t$ distinct from $B$.\label{q33}
    \end{enumerate}
\end{enumerate}

\begin{figure}[htb]
  \begin{center}
    \includegraphics[scale=0.7]{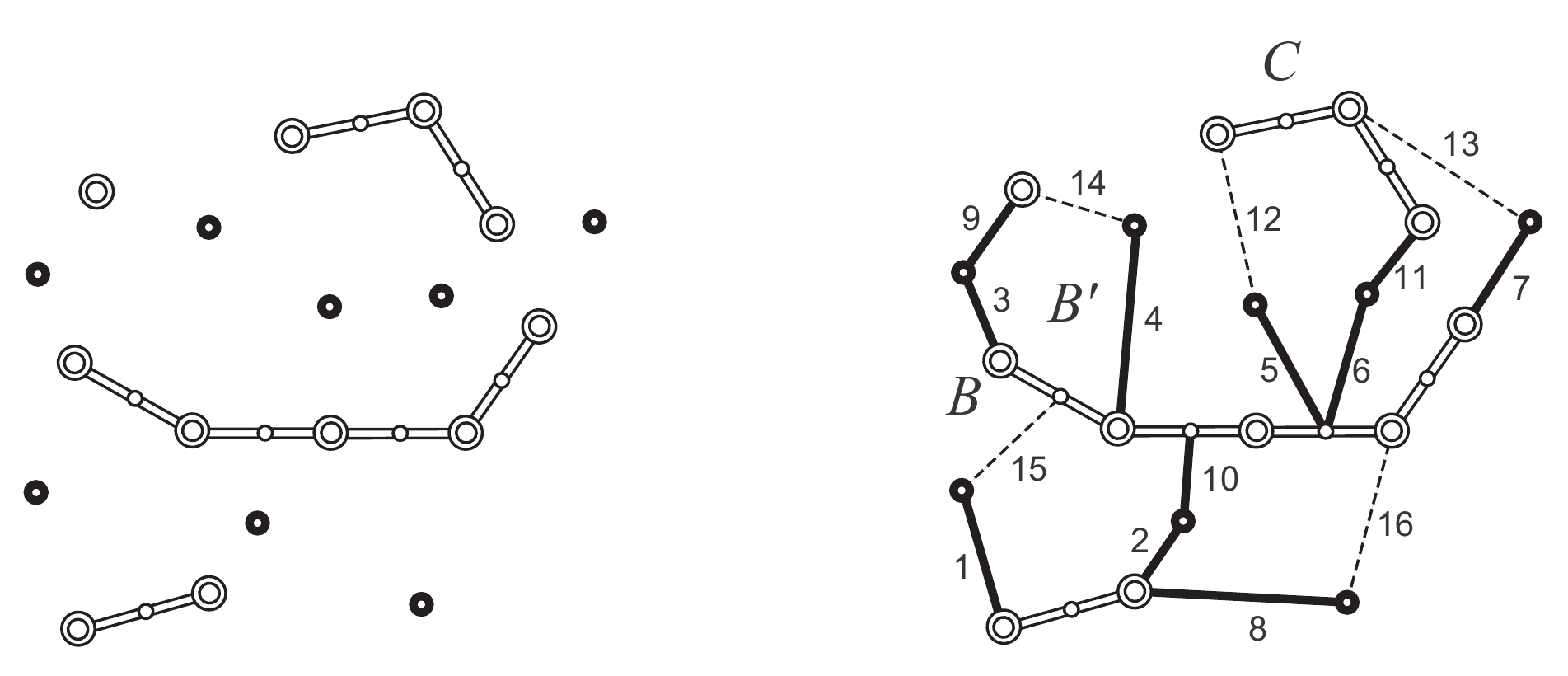}\\
  \end{center}
  \caption{An example construction by Function \texttt{BuildSES} of an augmented graph. The graph on the left is a given underlying graph with eight Steiner points, and serves as part of the input for the function. Numbers next to Steiner edges depict the order in which they were introduced to produce a $2$-connected graph}
  \label{figEgs19}
\end{figure}

Before formally presenting Function \texttt{BuildSES} we present an illustrative example in Figure \ref{figEgs19}, showing the basic steps that Function \texttt{BuildSES} performs. In this figure $F_0:=G_\mathrm{UN}$ is depicted in the left subfigure as a block-cut forest together with eight Steiner points. The sixteen Steiner edges introduced to produce a $2$-connected graph are depicted in the right subfigure by lines incident to Steiner points. Edge labels represent the order in which the Steiner edges were introduced by Function \texttt{BuildSES}, corresponding to the order of a reverse branching decomposition. Each Steiner edge $f_t$, where $1\leq t\leq 11$ (depicted by bold lines), joins two components of $F_{t-1}$. Each $f_t$ when $t>11$ (depicted by dashed lines) joins two distinct leaf-blocks of $F_{t-1}$.

Two aspects of the construction in Figure \ref{figEgs19} require further attention: notice firstly that $f_5$ and $f_6$ are incident to the same point. For this reason the block of $F_{12}$ containing edges $f_{5},f_{6},f_{11},f_{12}$ and component $C$ of $G_\mathrm{UN}$, is a leaf-block of $F_{12}$, and therefore the valid subsets in component $C$ are available for the formation of valid pairs at step $t=13$. In Function \texttt{BuildSES} the terminal endpoint of $f_t$ (if it has one) is chosen randomly within a valid subset, and therefore, since $f_5$ and $f_6$ are not necessarily adjacent in every possible execution of Function \texttt{BuildSES}, it is not clear that the endpoints for $f_{13}$ as given in Figure \ref{figEgs19} will be an available choice.

The second aspect that needs further consideration involves the ``order" in which terminal-neighbours of Steiner points occur within the interiors of degree-two block paths. In Figure \ref{figEgs19}, when $f_{14}$ is added to $F_{13}$, a new leaf-block $B'$ is formed containing edges $f_3,f_4,f_9,f_{14}$ and an isolated block of $G_\mathrm{UN}$. However, if the terminal endpoint of $f_{10}$ was contained in a block of $G_\mathrm{UN}$ that was closer to leaf-block $B$ of $G_\mathrm{UN}$ than the terminal endpoint of $f_4$, then $B'$ would not be a leaf-block of $F_{14}$ and therefore the valid subsets of $X$ contained in $B'$ would not be available at step $t=15$.

\begin{figure}[htb]
  \begin{center}
    \includegraphics[scale=0.5]{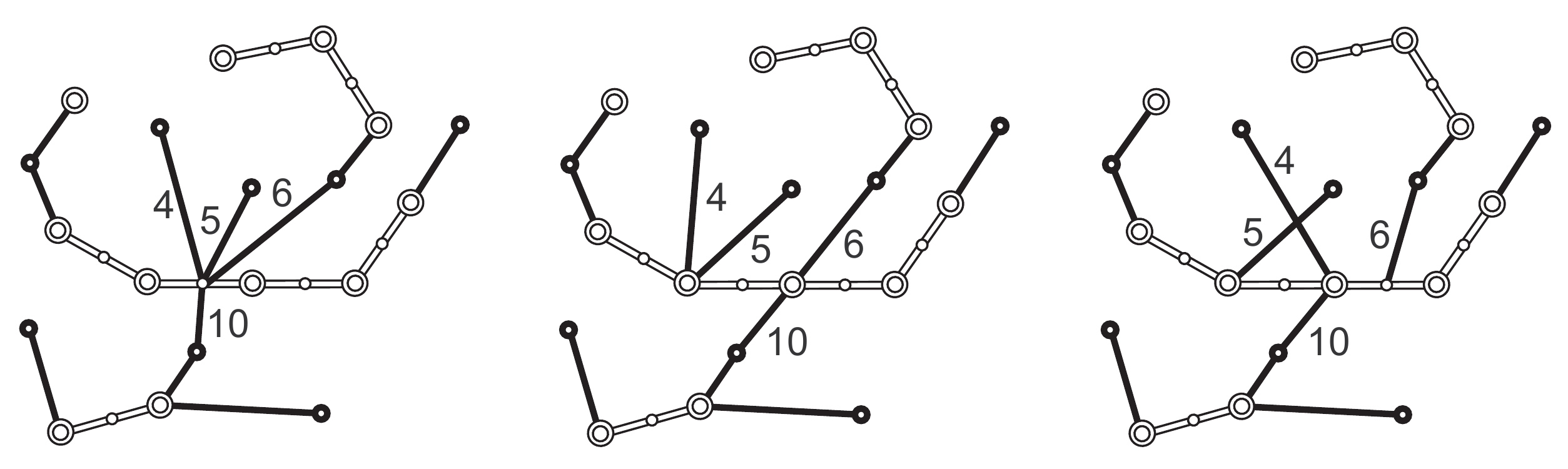}\\
  \end{center}
  \caption{Examples of graphs potentially output by Function \texttt{Order} on input $F_{t'}$ (from Figure \ref{figEgs19}) with $t'=11$}
  \label{figEgs20}
\end{figure}

Both the above aspects can easily be dealt with in $O(1)$ time in a subroutine which considers every possible relative ordering of terminal endpoints in all degree-two block path interiors of $G_\mathrm{UN}$, and by considering all cases that result when any two Steiner edges are incident to a common terminal. This is essentially because of our assumption that $b(G_\mathrm{UN})\leq \Delta k$ and that $k$ is constant. The subroutine which generates these new cases will henceforth be referred to as Function \texttt{Order}, and is always performed after step $t=t'$ of Function \texttt{BuildSES}, where $F_{t'}$ is connected but $F_{t'-1}$ is disconnected. The input for Function \texttt{Order} is the graph $F_{t'}$ and its output is the set of graphs that result after relocating terminal endpoints of Steiner edges for each new case. In the example of Figure \ref{figEgs19} we have $t'=11$, and three example graphs that Function \texttt{Order} could potentially output given input $F_{11}$ are illustrated in Figure \ref{figEgs20}. Since Function \texttt{Order} operates in an obvious way we do not provide further details on its structure.

We are now ready to present Function \texttt{BuildSES}, which explicitly constructs $\Lambda_0$ and also constructs exactly one $2$-connected representative $M_\mathcal{Q}$ for each $\mathcal{Q}\in\Lambda_0$. The procedure is recursive. We initiate the process by setting up the global variable $\Lambda_0:=\emptyset$ and then calling Function \texttt{BuildSES} with input: $F:=G_\mathrm{UN}$, the block-cut forest $F_\mathrm{BCF}$, and $t:=0$; the parameters $Y$ and $s$ initially have no values assigned to them.

\begin{algorithm}[H]
\SetAlgoLined
\NoCaptionOfAlgo
\KwIn{$F$, $F_{\mathrm{BCF}}$, $Y$, $s$, $t$}
\KwOut{$\Lambda_0$, $\{M_\mathcal{Q}:\mathcal{Q}\in\Lambda_0\}$}

FirstConnect $:= 0$
\tcc*[h]{\small{FirstConnect is the flag for Function \texttt{Order}}}\;
\eIf{$t>0$}
{
Let $x_t$ be a random element of $Y$\;
Let $s_t:=s$, let $Y_t:=Y$, and let $F_t$ be the graph that results by adding
edge $f_t=s_tx_t$ to $F$\;
If $F_t$ is connected but $F$ is disconnected then let FirstConnect $:=1$\;
Construct the block-cut forest $\mathrm{(}F_t\mathrm{)}_\mathrm{BCF}$\;
}
{
$F_t:=F$ and $(F_t)_\mathrm{BCF}:=F_\mathrm{BCF}$\;
}
\If{$t=\Delta k$ or there are no valid pairs for $F_t$}
{\If{$(F_t)_\mathrm{BCF}$ is an isolated vertex}
{Let $\mathcal{Q}:=\langle(s_1,Y_1),...,(s_t,Y_t)\rangle$ and let $M_\mathcal{Q}:=F_t$\;
Add $\mathcal{Q}$ to $\Lambda_0$\;
}
Exit\;
}
\eIf{\normalfont{FirstConnect} $=1$}
{
Run Function \texttt{Order} with input $F_t$ and output $\Psi$\;
}{
Let $\Psi:= \{F_t\}$\;
}
\For{every $F'\in\Psi$ and every valid pair $(s',Y')$ for $F'$}
{
Call Function \texttt{BuildSES} with input: $F'$, $F'_\mathrm{BCF}$, $Y'$, $s'$, $t+1$\;
}
\caption{\textbf{Function} BuildSES \label{BuildSES}}
\end{algorithm}

\begin{theorem}\label{mainN}Every $N\in \mathcal{N}_0$ has a Steiner endpoint sequence in $\Lambda_0$.
\end{theorem}
\begin{proof}
Let $\langle e_1,...,e_{p_N}\rangle$ be a reverse branching decomposition of an arbitrary $N\in \mathcal{N}_0$, which exists by Proposition \ref{anyBranch}. Let $t'$ be the smallest $t$ such that $N_t^-:=G_\mathrm{UN}\cup \{e_j:1\leq j\leq t\}$ is connected. Since the valid subsets of $X$ partition $X$, and Function \texttt{BuildSES} considers every valid pair, we may assume that the first $t'$ pairs of some $\mathcal{Q}\in \Lambda_0$ are a Steiner endpoint sequence for $N_{t'}^-$. Therefore, for $t\leq t'$, each $e_t=s_tx_t'$, for some $x_t'\in Y_t$. When $t>t'$, those valid subsets of $X$ which are available for forming valid pairs do not necessarily partition $X$. However, we demonstrate that, due to Function \texttt{Order}, some sequence of edge additions $f_{1},...,f_{p_N}$ during the execution of Function \texttt{BuildSES} has the following property at $t''=P_N$. Let Property ${P}(t'')$ be the property that for every $t\in \{t',...,t''\}$ there is a one-to-one correspondence between the leaf blocks of $F_t$ and $N_t^-$ such that every corresponding pair shares a subset of terminals from a valid subset of $X$. Since Function \texttt{BuildSES} considers all possible leaf blocks and all valid subsets that intersect each leaf-block, showing that Property ${P}(p_N)$ holds will prove the theorem. To show this we employ induction on $t$, with the base case $t=t'$.

We first show that Property $P(t')$ holds. Notice that every leaf-block ${B}$ of $F_{t'}$ has one of the following forms: ${B}$ is a leaf-block of $G_\mathrm{UN}$; ${B}$ is an isolated block of $G_\mathrm{UN}$ containing at least two vertices; or ${B}$ is a unique edge of $F_{t'}$ incident to an isolated vertex of $G_\mathrm{UN}$. In fact, a leaf block ${B}'$ of $G_\mathrm{UN}$ is a leaf-block of $F_{t'}$ if and only if no Steiner edges are incident to the interior of ${B}'$; and an isolated block $W$ of $G_\mathrm{UN}$, containing at least two vertices, is a leaf-block of $F_{t'}$ if and only if all Steiner edges incident to $W$ are incident to the same point. Therefore, since Function \texttt{Order} considers every possible ordering of terminal endpoints in all degree-two block path interiors, and considers all cases that result when any two Steiner edges are incident to a common terminal in $F_{t'}$, Function \texttt{BuildSES} constructs some $F_{t'}$ which has identical leaf-blocks to $N_{t'}^-$. Therefore Property $P(t')$ holds.

Let $P_F,P_N$ be any paths in such an $F_{t'}$ and in $N_{t'}^-$ respectively, connecting a point in a leaf-block $B$ to a point in a distinct leaf-block $B'$. Observe that for every terminal $x$ on $P_N$ there exists a terminal $y$ on $P_F$ such that $x$ and $y$ belong to the same valid subset of $X$ (and vice-versa); we say that Property $D(t)$ holds at $t=t'$.

Suppose that for some $t\geq t'$ Property $P(t)$ and Property $D(t)$ are satisfied. Observe that when introducing an edge between the interiors of two distinct leaf-blocks $B_0,B_0'$ of $F_t$, a new block $\widehat{B}$ is formed which contains all vertices of the blocks in the block path connecting $B_0$ and $B_0'$. Due to Function \texttt{Order} and by the discussion accompanying Figure \ref{figEgs19}, we may assume that block $\widehat{B}$ is a leaf-block of $F_{t+1}$ if and only if the corresponding block in $N_{t+1}^-$ is a leaf-block. Since Property $D(t)$ is true it follows that, if $B$ and $B'$ are corresponding leaf-blocks in $F_{t+1}$ and $N_{t+1}^-$ respectively, then every valid subset $X'$ that intersects $V(B)$ also intersects $B'$. Therefore Properties $P(t+1)$ and $D(t+1)$ are true. This proves the theorem.
\end{proof}

Finally we look at the complexity of Function \texttt{BuildSES}. The depth of the recursion tree in Function \texttt{BuildSES} is clearly at most $\Delta k$. Also, since it is assumed that $b(G_\mathrm{UN})\leq \Delta k$, every $F_t$ contains at most $O(1)$ valid pairs. Therefore the maximum degree of a node of the recursion tree is also of constant order.  Constructing a block-cut forest takes $O(n^2)$ time. All other nodes of the recursion tree, as well as Function \texttt{Order}, run in constant time. Therefore we have shown the following:

\begin{proposition}Function \normalfont{\texttt{BuildSES}} \textit{runs in $O(n^2)$-time}.
\end{proposition}

For the rest of this paper we assume that all Steiner edges and chord paths of $M_\mathcal{Q}$ for any $\mathcal{Q}\in\Lambda_0$ are critical and that $M_\mathcal{Q}$ contains no linked sets, since, within time $O(n^2)$, any $\mathcal{Q}$ such that $M_\mathcal{Q}$ has one of these components can be removed from $\Lambda_0$.

\subsubsection{Ensuring $2$-connectivity}\label{makeCon}
Not every representative of $G(\mathcal{Q})$, where $\mathcal{Q}\in\Lambda_0$, is necessarily $2$-connected. Therefore, to satisfy Property (D), we present in this section Function \texttt{2Connect}, which constructs a $2$-connected representative $G_\mathrm{OPT}^2(\mathcal{Q})$ of $G(\mathcal{Q})$ such that $\ell_\mathrm{max}(G_\mathrm{OPT}^2(\mathcal{Q}))=\ell_\mathrm{max}(G_\mathrm{OPT}(\mathcal{Q}))$. We will need the following lemma.

\begin{figure}[htb]
  \begin{center}
    \includegraphics[scale=0.65]{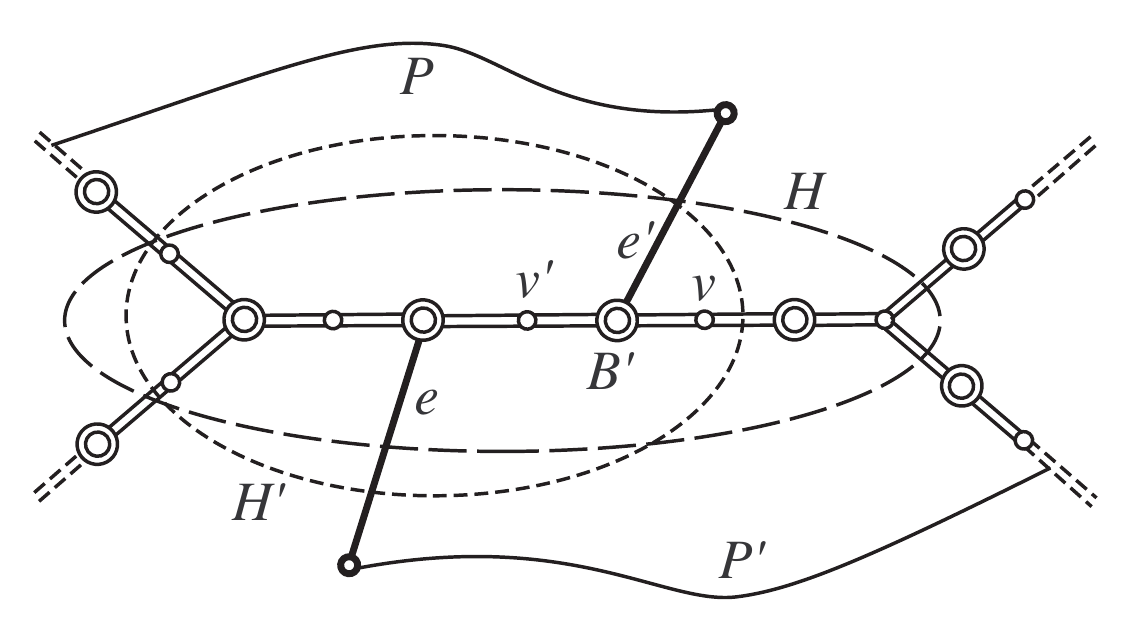}\\
  \end{center}
  \caption{Proof of Lemma \ref{bcfIn}. The curved lines represent paths in $N$. The long and short dashed ellipses encircle $H$ and $H'$ respectively}
  \label{figEgs14}
\end{figure}

\begin{lemma}\label{bcfIn}Let $G$ be a connected subgraph of an augmented network $N$, where $N$ contains no linked sets and $G$ contains all edges of $G_\mathrm{UN}$. Let $B$ be a leaf-block of $G$, let $x\in\mathrm{int}(B)$ be a terminal, and let $X'$ be the valid subset of $X$ containing $x$. Then $X'\subseteq B$.
\end{lemma}
\begin{proof}
Recall that $X'$ has one of four possible forms. Clearly the lemma is true if $X'$ is a singleton containing a vertex of $V_\mathrm{CV}$. So suppose next that $X'$ is the interior of a leaf-block or isolated block $\widehat{B}$ of $G_\mathrm{UN}$. Since blocks are maximal $2$-connected subgraphs, $\widehat{B}$ must be contained in a block $\widehat{B}'$ of $G$. Since $x$ is in the interior of $B$, no other blocks of $G$ contain $x$. Therefore $B=\widehat{B}'$.

We only need to consider one more case. Suppose that $X'$ is the interior of a degree-two block path $H$ of $G_\mathrm{UN}$ (see Figure \ref{figEgs14}). We assume that $H$ contains at least two blocks that are of degree two in $(G_\mathrm{UN})_\mathrm{BCF}$, for otherwise the result follows similarly to the case when $X'$ is the interior of a leaf-block of $G_\mathrm{UN}$. Let $H'$ be the largest subgraph of $H$ such that $H'\subseteq B$, and suppose that $\mathrm{int}(H)\nsubseteq H'$. Let $v$ be a cut vertex of $H$ such that $v$ belongs to at least two distinct blocks of $G$, exactly one of which has an interior intersecting $H'$. Then $v$ is a cut-vertex of $G$. Since $B$ is a leaf-block, $v$ is unique, and $H'$ includes all vertices of $\mathrm{H}$ in the component of $G-v$ containing $H'$.

Let $v'$ be a cut vertex of $G_\mathrm{UN}$ contained in the same block $B'$ of $H$ as $v$. Then, since $v'$ is contained in $\mathrm{int}(B)$, there exists a path $P$ in $B$ connecting $v'$ and a vertex in $B'-v'$, such that $P$ contains edges not in $G_\mathrm{UN}$. Specifically, an end-edge of $P$, say $e'$, which is an external Steiner edge, is incident to a vertex of $B'-v'$. Similarly, since $v$ is not a cut-vertex of $N$, there is a path $P'$ in $N$ with a Steiner end-edge $e$ connecting a terminal of $G-B$ to a terminal of $H'$. But then $\{e,e'\}$ is a linked set in $N$, which contradicts the fact that $N\in \mathcal{N}_0$.
\end{proof}

Now suppose that some representative of $G(\mathcal{Q})$, where $\mathcal{Q}\in\Lambda_0$ and $\mathcal{Q}=\langle(s_1,Y_1),...,(s_\rho,Y_\rho)\rangle$, is not $2$-connected. Note that all representatives of $G(\mathcal{Q})$ are connected, since the terminal endpoint of any given external Steiner edge lies in a fixed component of $G_\mathrm{UN}$. Let the Steiner edges of $M_\mathcal{Q}$ be $f_i=s_iu_i$, for $1\leq i\leq \rho$, where $u_i\in\{s_i',x_i\}$. Let $$\mathcal{Q}_b=\langle (s_1,Y_1),...,(s_b,Y_b),(s_{b+1},\{u_{b+1}\}),...(s_{\rho},\{u_{\rho}\})\rangle,$$ where $b\geq 0$ is the largest integer such that every representative of $G(\mathcal{Q}_b)$ is $2$-connected. Observe that the unique representative of $G(\mathcal{Q}_0)$ is $M_\mathcal{Q}$.

Let $N_{\mathcal{Q}_{b+1}}$ be a representative of $G(\mathcal{Q}_{b+1})$ that is not $2$-connected, where $E_\mathrm{SE}(N_{\mathcal{Q}_{b+1}})=\{e_1,...,e_\rho\}$. Then $Y_{b+1}$ contains at least two terminals, for otherwise $\mathcal{Q}_b=\mathcal{Q}_{b+1}$. By Lemma \ref{bcfIn} it follows that the valid subset $Y_{b+1}$ is contained in a leaf block of $N_{\mathcal{Q}_{b+1}}-e_{b+1}$, say $B_1$. Since $N_{\mathcal{Q}_{b+1}}$ is not $2$-connected, this means that, in $N_{\mathcal{Q}_{b+1}}$, the edge $e_{b+1}$ is incident to the cut-vertex, say $v$, of $N_{\mathcal{Q}_{b+1}}-e_{b+1}$ contained in $B_1$. Note that $v$ cannot be a Steiner point nor a vertex of $V_\mathrm{CV}$. If $v$ is incident to terminal edges in both blocks of $N_{\mathcal{Q}_{b+1}}-e_{b+1}$ containing $v$, then $v$ is also a cut vertex of $G_\mathrm{UN}$. This contradicts Lemma \ref{bcfIn}. Therefore, the set of Steiner edges incident to $v$ comprise an edge-cut of $N_{\mathcal{Q}_{b+1}}$. Let $\Gamma'=\{e_{j}:j\in J\}$, for some index set $J$, be a minimal subset of this edge-cut. Since, in the candidate type $G(\mathcal{Q})$, external Steiner edge $h_i$ is incident to a vertex in a fixed component of $G_\mathrm{UN}$, the set $\Gamma =\{h_j:j\in J\}$ of labelled edges must be an edge cut in all the representatives of $G(\mathcal{Q})$. We therefore refer to $\Gamma$ as an \textit{edge cut of} $G(\mathcal{Q})$.

A minimal edge-cut $\Gamma$ of $G(\mathcal{Q})$, consisting of external Steiner edges only, is referred to as a \textit{potential cut} for $G(\mathcal{Q})$ if $\displaystyle\bigcap_{j\in J}Y_{j}\neq \emptyset$.

The next lemma now follows from the above discussion.

\begin{lemma}\label{notCon}If some representative of $G(\mathcal{Q})$ is not $2$-connected then there exists a potential cut for $G(\mathcal{Q})$.
\end{lemma}

\begin{lemma}\label{least2}Let $T$ be the Steiner topology of $G(\mathcal{Q})$. At least two components of $T$ contribute edges to any potential cut $\Gamma$ for $G(\mathcal{Q})$.
\end{lemma}
\begin{proof}
Suppose to the contrary that all edges of $\Gamma$ belong to the same component, say $T'$, of $T$. Since $M_\mathcal{Q}-T'$ is connected, all external edges of $T'$ lie in $\Gamma$. Let $X'=\displaystyle\bigcap_{j\in J}Y_{j}\neq \emptyset$. Any pair of distinct valid subsets of $X$ are disjoint, therefore $X'$ is a valid subset of $X$. If $X'$ is a singleton then $X'$ is a cut-vertex of $M_\mathcal{Q}$, which is a contradiction. If $X'$ is contained in a block of $G_\mathrm{UN}$ then $T'$ can be removed from $M_\mathcal{Q}$ without losing $2$-connectivity (in other words, in this case there exist non-critical edges or non-critical chord-paths in $M_\mathcal{Q}$). Therefore, suppose that $X'$ is the interior of a degree-two block path $H$ of $G_\mathrm{UN}$, and that $H$ contains at least two blocks of degree two in $H_\mathrm{BCF}$. Let $x,y$ be distinct terminal endpoints of edges of $\Gamma$, contained in such degree-two blocks $B,B'$ of $H$ respectively, such that the number of blocks in the block path between $B$ and $B'$ is a maximum.

Suppose first that $\Gamma$ contains at least three elements, and let $z$ be a distinct (from $x,y$) terminal endpoint of an edge in $\Gamma$. Let $P$ be the path in $T'$ connecting $x$ and $y$, and let $P'$ be the shortest path in $T'$ connecting $z$ to a Steiner point on $P$. Since there exists a path $P_0$ in $G_\mathrm{UN}$ connecting $x$ and $y$, such that $P_0$ contains $z$, path $P'$ is a chord of the cycle $P\cup P_0$. This contradicts Lemma \ref{critStuff}. Finally, suppose that $\Gamma$ only contains two edges. Then $T'$ is a path with interior vertices consisting of degree-two Steiner points only. But then $B$ and $B'$ must lie in distinct leaf-blocks of the block path $M_\mathcal{Q}-T'$, for otherwise $T'$ would be a degree-two chord path. This contradicts Lemma \ref{bcfIn}.
\end{proof}

\begin{figure}[htb]
  \begin{center}
    \includegraphics[scale=0.55]{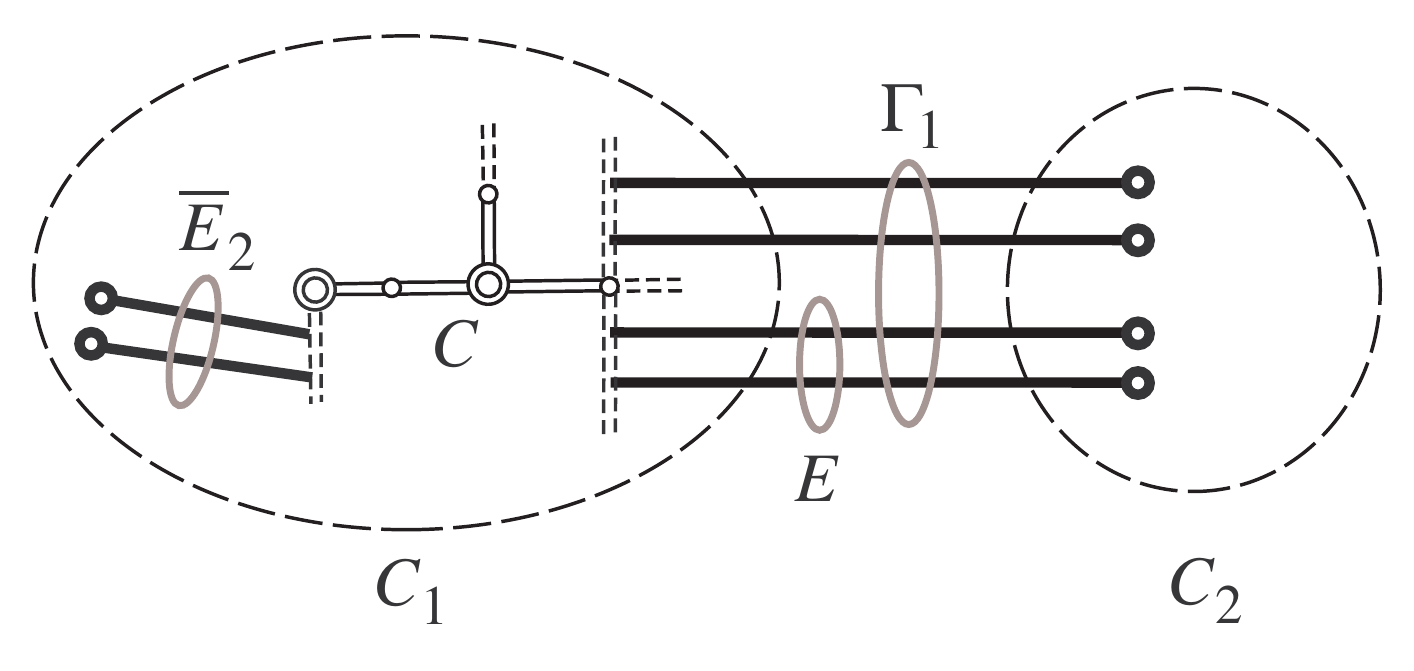}\\
  \end{center}
  \caption{Proof of Lemma \ref{disEd}. The Steiner edges in $\Gamma_1$ and $\Gamma_2$ are depicted as bold lines}
  \label{figEgs17}
\end{figure}

As a consequence of the next lemma the set of potential cuts for $G(\mathcal{Q})$ is unique and every pair of distinct potential cuts for $G(\mathcal{Q})$ is disjoint.

\begin{lemma}\label{disEd}Let $N$ be an augmented network. Suppose that $\Gamma_1,\Gamma_2$ are distinct minimal edge-cuts of $N$, with both sets containing only external Steiner edges. If every edge in $\Gamma_1\cup\Gamma_2$ is incident to the same component of $G_\mathrm{UN}$ then $\Gamma_1$ and $\Gamma_2$ are disjoint.
\end{lemma}
\begin{proof}
Suppose that every edge in $\Gamma_1\cup\Gamma_2$ is incident to the same component, say $C$, of $G_\mathrm{UN}$, and that $E= \Gamma_1\cap\Gamma_2\neq \emptyset$; see Figure \ref{figEgs17} where grey ellipses are used to highlight particular subsets of edges. Note that $E\neq \Gamma_i$ for any $i$, since the $\Gamma_i$ are minimal. Therefore, since $N$ is either a block or a block path, $N-\Gamma_1$ consists of exactly two connected components, say $C_1,C_2$. Suppose without loss of generality that $C$ is contained in $C_1$. Since $\Gamma_1$ is an edge-cut, and all edges of $\Gamma_2$ are incident to $C$, all edges of $\overline{E}_2:=\Gamma_2-E$ must be contained in $C_1$. Let $N'=N-\overline{E}_2$. Then $E$ is an edge-cut of $N'$. But $C$ and $C_2$ are connected in $N-\Gamma_2$, and consequently also in $N'$. This contradicts the fact that $E$ is an edge-cut of $N'$, since all endpoints of $E$ lie in $C\cup C_2$.
\end{proof}

Since, by assumption, some representative of $G(\mathcal{Q})$ is not $2$-connected, it is possible that $G_{\mathrm{OPT}}(\mathcal{Q})$ is not $2$-connected. However, since $M_\mathcal{Q}$ is $2$-connected, there exists a graph $G_\mathrm{OPT}^2(\mathcal{Q})$ such that $G_\mathrm{OPT}^2(\mathcal{Q})$ is a cheapest $2$-connected representative of $G(\mathcal{Q})$. For any potential cut $\Gamma$ for $G(\mathcal{Q})$, and any edge $e\in\Gamma$, let $x(e)$ be the terminal endpoint of $e$ in $G_{\mathrm{OPT}}(\mathcal{Q})$.

\begin{lemma}\label{atVertex}There exists a cheapest $2$-connected representative $G_\mathrm{OPT}^2(\mathcal{Q})$ of $G(\mathcal{Q})$, a potential cut $\Gamma$ for $G(\mathcal{Q})$, and an edge $e\in\Gamma$, such that the terminal endpoint of $e$ in $G_\mathrm{OPT}^2(\mathcal{Q})$ is $x(e)$.
\end{lemma}
\begin{proof}
Suppose that the lemma is not true and let $G_\mathrm{OPT}^2(\mathcal{Q})$ be a cheapest $2$-connected representative of $G(\mathcal{Q})$. Recall that, when constructing $G_{\mathrm{OPT}}(\mathcal{Q})$ as in Section \ref{jusdoit}, the optimal endpoints for every Steiner edge are found. As observed by Bae et al. \cite{bae1}, each component of the Steiner topology $T$ of $G(\mathcal{Q})$ can be independently dealt with. Let $T'$ be a component of $T$ such that $T'$ contains an edge of some $\Gamma$. Suppose that the external Steiner edges of $G_\mathrm{OPT}^2(\mathcal{Q})$ are $s_1x_1',...,s_px_p'$ and the external Steiner edges of $T'$ in $G_\mathrm{OPT}$ are $\{s_ix_i'':i\in I\}$ for some index set $I$. Let $\mathcal{Q}'=\langle(\{s_1\},\{x_1'\}),...,(\{s_p\},\{x_p'\})\rangle$ and let $\mathcal{Q}''$ be the sequence that results by replacing $(\{s_i\},\{x_i'\})$ by $(\{s_i\},\{x_i''\})$ in $\mathcal{Q}'$ for every $i\in I$. Since the lemma is assumed to be false, it follows from Lemma \ref{least2} that the edges of any minimal external edge-cut are not all incident to the same point in $G_\mathrm{OPT}(\mathcal{Q}'')$; therefore $G_\mathrm{OPT}(\mathcal{Q}'')$ is $2$-connected and $G_\mathrm{OPT}^2(\mathcal{Q}'')$ exists. Now $\ell_\mathrm{max}(G_\mathrm{OPT}^2(\mathcal{Q}))=\ell_\mathrm{max}(G_\mathrm{OPT}^2(\mathcal{Q}''))$, since in $G_{\mathrm{OPT}}(\mathcal{Q})$ the edges of $T'$ have optimal bottleneck length. The lemma follows.
\end{proof}

Of course, since $G_\mathrm{OPT}^2(\mathcal{Q})$ is $2$-connected, some edge $e$ in $\Gamma$ is not incident to $x(e)$.

Suppose that there are $q$ potential cuts for $G(\mathcal{Q})$. Using the above results, we now modify $\mathcal{Q}$ to a set of new sequences $\{\mathcal{Q}(i,q)\}$ so that every representative of every $G(\mathcal{Q}(i,q))$ is $2$-connected. This directly leads to a constructive method for finding $G_\mathrm{OPT}^2(\mathcal{Q})$, namely Function \texttt{2Connect}, which we describe next.

As before, let $Q:=\langle (s_1,Y_1),...,(s_\rho,Y_{\rho})\rangle$, and denote by $h_i$ the labelled edge of $G(\mathcal{Q})$ with endpoints in $\{s_i\},Y_{i}$ respectively. For non-negative integers $i,j$, let $$\mathcal{Q}(i,j)=\langle (s_1,Y_1(i,j)),...,(s_\rho,Y_\rho(i,j))\rangle$$ denote a Steiner endpoint sequence. Let $\{\Gamma_1,...,\Gamma_q\}$ be the set of potential cuts for $G(\mathcal{Q})$.\\

\newcommand{\forcond}{$j:=1$ \KwTo $q$}

\begin{algorithm}[H]
\SetAlgoLined
\NoCaptionOfAlgo
\KwIn{$\mathcal{Q}$}
\KwOut{$G_\mathrm{OPT}^2(\mathcal{Q})$}
$\Lambda':=\emptyset$\\
$i:=1$\;
\For{every set $\{r_1,...,r_q\}=\{1,...,q\}$}
{
\For{every $(h_{t_1},...,h_{t_q})\in \Gamma_{r_1}\times...\times\Gamma_{r_q}$ and $(h_{b_1},...,h_{b_q})\in \Gamma_{r_1}\times...\times\Gamma_{r_q}$ such that $t_c\neq b_c$ for any $c$}{
$\mathcal{Q}(i,1):=\mathcal{Q}$\;
\For{\forcond}{
Let $G^j=G_{\mathrm{OPT}}(\mathcal{Q}(i,j))$\;
Let $x({h_{t_j}})$ be the terminal endpoint of $h_{t_j}$ in $G^j$\;
Let $Y_{t_j}(i,j):=\{x({h_{t_j}})\}$\;
Let $Y_{b_j}(i,j):=Y_{b_j}(i,j)-\{x({h_{t_j}})\}$\;
}
Add $\mathcal{Q}(i,q)$ to $\Lambda'$\;
$i:=i+1$\;
}
}
Let $G_\mathrm{OPT}^2(\mathcal{Q})$ be a cheapest member of $\{G_\mathrm{OPT}(\mathcal{Q}'):\mathcal{Q}'\in\Lambda'\}$\;
\caption{\textbf{Function} 2Connect \label{mkLambda1}}
\end{algorithm}

It follows from repeated application of Lemma \ref{atVertex} that Function \texttt{2Connect} correctly computes $G_\mathrm{OPT}^2(\mathcal{Q})$. Note also that the time-complexity of Function \texttt{2Connect} is the same as the complexity of constructing $G_\mathrm{OPT}(\mathcal{Q})$ (as given in Theorem \ref{consOp}), since the cardinality of $\Lambda'$ is a function of $k$ only. Also, the potential cuts of any $G(\mathcal{Q})$ can found in $O(n^2)$ time as a preprocessing step.

\begin{figure}[htb]
  \begin{center}
    \includegraphics[scale=0.6]{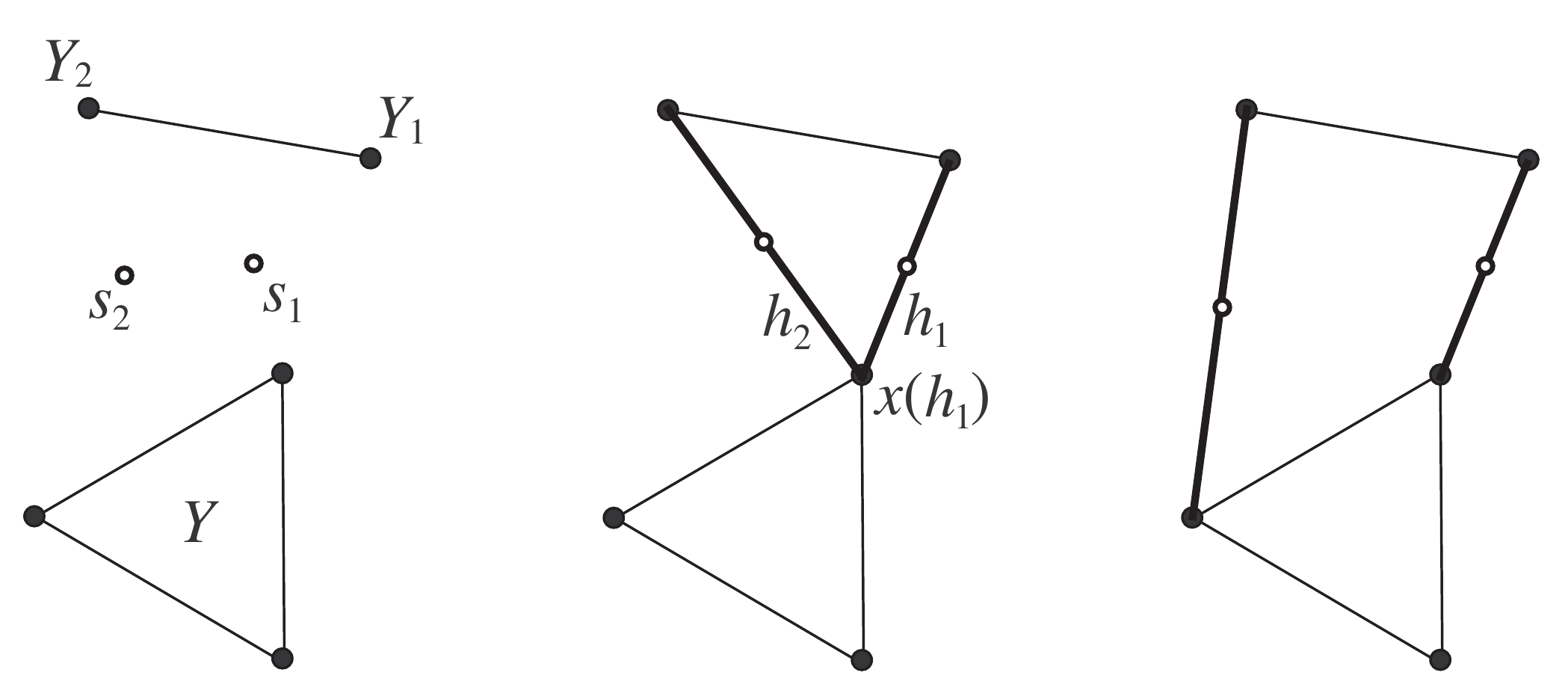}\\
  \end{center}
  \caption{An example of a construction by Function \texttt{2Connect}}
  \label{figEgs25}
\end{figure}

In Figure \ref{figEgs25} we illustrate some aspects of Function \texttt{2Connect} with input sequence $\mathcal{Q}=\langle(s_1,Y),(s_2,Y),(s_1,Y_1),(s_2,Y_2)\rangle$. The first subfigure depicts $G_\mathrm{UN}$ and a corresponding partition of the the terminal-set into valid subsets, $Y_1,Y_2,Y$. The second subgraph depicts $G_\mathrm{OPT}(\mathcal{Q})$ and the node $x({h_1})=x({h_2})$, where $\{h_1,h_2\}$ is a potential cut. The final subgraph shows $G_\mathrm{OPT}^2(\mathcal{Q})$, which was the cheapest of two graphs constructed by Function \texttt{2Connect} in this case (the other graph is visualised by switching the terminal endpoints of $h_1$ and $h_2$ in the third subfigure).

\subsubsection{Augmented networks with linked sets}\label{withLink}
In this section we show how we deal with augmented networks that do contain linked sets. Essentially we show that every augmented graph $\widehat{N}$ can be obtained from a graph $N\in\mathcal{N}_0$ by ``splitting" internal Steiner edges. This consists of a process whereby some internal Steiner edges of $\widehat{N}$ are replaced by linked sets. We then show how the Steiner endpoint sequence for which $\widehat{N}$ is a representative is modified in order to accommodate the new linked edges. For this purpose we introduce the concept of ``markers". For every internal Steiner edge $h$ of $\widehat{N}$ that we are to split, we basically choose (or ``mark") a block $B$ of $\widehat{N}-h$. This divides the terminals of $\widehat{N}-h$ into two overlapping parts, namely the terminals in the blocks to the left and the right of $B$ (for some orientation, where each part includes $B$). Each of these parts serves as a valid subset for one of the two edges in the new linked set. There may be as many as $n$ distinct blocks in $\widehat{N}-h$, and therefore choosing the optimal marker cannot be done in constant time. However, in the 2 Bottleneck algorithm we perform a binary search on the set of markers in order to reduce the complexity of finding an optimal solution.

Recall that $\mathcal{N}$ is the class of \textit{all} augmented networks (containing $G_\mathrm{UN}$), such that every $N\in\mathcal{N}$ has at most $\Delta k$ Steiner edges, and such that the chord paths and Steiner edges of every member of $\mathcal{N}$ are critical. We begin by stating a few more results related to the concept of linked sets.

A general property of linked sets -- that can be seen from Figure \ref{figEgNew2} and is easily demonstrated -- is that there is an implicit ordering of the terminal endpoints of the two Steiner edges with respect to the blocks of $N-E_{\mathrm{LINK}}$, for any $N\in\mathcal{N}$ and any linked set $E_{\mathrm{LINK}}$. Suppose that $E_{\mathrm{LINK}}=\{g_{1,1},g_{1,2}\}$ and that the block path of $N-E_{\mathrm{LINK}}$ is $B_1,...,B_{q}$ for some $q>2$. Without loss of generality let $g_{1,1}$ be the unique edge of $E_{\mathrm{LINK}}$ with Steiner endpoint contained in the interior of $B_1$. Let $B_{-1}=B_{q+1}=\emptyset$. For every $i\in\{1,...,q\}$, let $U_{i,1}=B_i-(V(B_{i-1})\cap V(B_i))$, and let $U_{i,2}=B_i-(V(B_{i+1})\cap V(B_i))$. Let $s_1,s_2$ be the Steiner endpoints of $g_{1,1},g_{1,2}$ respectively. Now consider a set of vertex pairs $\{g_{i,j}:1\leq j\leq 2\}$, which includes the edges $g_{1,1},g_{1,2}$, where each $g_{i,j}=(s_j,y_{i,j})$ and each $y_{i,j}$ is a terminal of $N$. For any $i$ let $a_{i}\in\{1,...,q\}$ be the index such that $y_{i,1}\in U_{a_{i},1}$, and let $a_{i}'\in\{1,...,q\}$ be the index such that $y_{i,2}\in U_{a_{i}',2}$; see Figure \ref{figEgNew2}. For any $i$, if $a_{i}'\leq a_{i}$, where $1<a_{i}'<q$ and $1<a_{i}<q$, then we write $y_{i,2}\preccurlyeq y_{i,1}$ (note that the indices of the leaf-blocks are excluded). For any $i$ let $N^-=N-E_{\mathrm{LINK}}+g_{i,1}+g_{i,2}$.

\begin{lemma}[\cite{brazil2}]\label{isLink}$N^-$ is a member of $\mathcal{N}$ and $\{g_{i,1},g_{i,2}\}$ is a linked set of $N^-$, if and only if $y_{i,2}\preccurlyeq y_{i,1}$.
\end{lemma}

Let $N\in\mathcal{N}_0$ and let $e=s_is_i'$ be an internal Steiner edge of $N$. We say that $e$ is \textit{splittable} if there exists a pair of terminals $y_1,y_2$ in $N$ such that $y_2\preccurlyeq y_1$ for some orientation of $N-e$. Let $N'=N-e+s_iy_1+s_i'y_2$. Then $N'$ is called a \textit{split} of $N$ with respect to $e$. We generalise this definition as follows.

\noindent\textbf{Definition.} \textit{Let $E_\mathrm{\,I}=\langle e_1,...,e_{|E_\mathrm{\,I}|}\rangle$ be a sequence of internal Steiner edges of $N\in\mathcal{N}_0$. Let $\widehat{N}(0),...,\widehat{N}(\vert E_\mathrm{\,I}\vert)$ be a sequence, with $\widehat{N}(0)=N$, such that, for every $i\in \{1,...,\vert E_\mathrm{\,I}\vert\}$, the graph $\widehat{N}(i)$ is a split of $\widehat{N}(i-1)$ with respect to $e_i$. Then $\widehat{N}:=\widehat{N}(\vert E_\mathrm{\,I}\vert)$ is a \textbf{split} of $N$ with respect to $E_\mathrm{\,I}$.}

We employ the following notation in the next proposition. Let $\widehat{N}$ be any graph in $\mathcal{N}$. For any $t\geq 1$ let $Z=\langle\{g_1,g_1'\},...,\{g_t,g_t'\}\rangle$ be a sequence of pairs of external Steiner edges of $\widehat{N}$, where each $g_i=s_iy_i$ and each $g_i'=s_i'y_i'$, where $s_i,s_i'$ are Steiner points, and where $y_i,y_i'$ are terminals. For every $1\leq j\leq t$, let $\widehat{N}_Z(t-j)=\widehat{N}-\{g_i:1\leq i\leq j\}-\{g_i':1\leq i\leq j\}+\{s_is_i':1\leq i\leq j\}$.

\begin{proposition}For every $\widehat{N}\in \mathcal{N}$ containing at least one linked set there exists a graph $N\in \mathcal{N}_0$ and sequence $E_\mathrm{\,I}$ of internal Steiner edges of $N$ such that $\widehat{N}$ is a split of $N$ with respect to $E_\mathrm{\,I}$.
\end{proposition}
\begin{proof}
Let $\widehat{N}\in\mathcal{N}$. We select the elements of sequence $Z$ as follows: let $\{g_1,g_1'\}$ be a linked set of $\widehat{N}$ (if no linked set exists then the theorem is proven). Since $s_1\in \mathrm{int}(B_1(\{g_1,g_1'\})\vert_{\widehat{N}})$ and $s_1'\in \mathrm{int}(B_{2}(\{g_1,g_1'\})\vert_{\widehat{N}})$ for some orientation, by Lemma \ref{joinB}, the graph $\widehat{N}(t-1)$ is $2$-connected. It is simple to verify that no non-critical Steiner edges or degree two chord paths were formed during the transformation from $\widehat{N}$ to $\widehat{N}(t-1)$, since all cycles in $\widehat{N}(t-1)$ that do not occur in $\widehat{N}$ include edge $s_1s_1'$. Therefore $\widehat{N}(t-1)\in \mathcal{N}$. Since $\{g_1,g_1'\}$ is a linked set, it follows from Lemma \ref{isLink} that $y_1'\preccurlyeq y_1$. Next, let $\{g_2,g_2'\}$ be a linked-set of $\widehat{N}(t-1)$ (once again, if no such set exists then the theorem follows). We perform the same process for $\{g_2,g_2'\}$ to arrive at the graph $\widehat{N}(t-2)$, and repeat until there are no more linked sets to be found. Suppose that the resultant graph containing no linked sets occurs after $q$ steps; let $t=q$, and let $N=\widehat{N}(0)$. Then $N\in\mathcal{N}_0$ since $N$ has no linked sets, and all Steiner edges and chord paths are critical. Also, $\widehat{N}(0),...,\widehat{N}(q)$ satisfies the property that $\widehat{N}(i)$ is a split of $\widehat{N}(i-1)$ for all $i\in\{1,...,q\}$. Therefore $\widehat{N}$ is a split of $N$ with respect to $E_\mathrm{\,I}:=\langle s_qs_q',...,s_1s_1' \rangle$, and the theorem follows.
\end{proof}

Now that we have shown that every augmented network $\widehat{N}$ is a split of a graph in $N\in\mathcal{N}_0$, we describe how the Steiner endpoint sequence for $N$, say $\mathcal{Q}$, is converted into a Steiner endpoint sequence for $\widehat{N}$. We do this by first constructing a certain canonical representative (referred to as $M(a)$ below) for $\mathcal{Q}$. The purpose of constructing the canonical representative is obtain a fixed block path $M(a)-h_a$ for each internal Steiner edge $h_a$ of $N$, so that the afore-mentioned binary search that is to be performed on the set of ``markers" is well defined.

In what follows we will be considering various specific representatives of $G(\mathcal{Q})$ for some $\mathcal{Q}\in \Lambda_1$. The symbol $h_j$, which is the labelled edge of $G(\mathcal{Q})$ corresponding to $(s_j,Y_j)$ in $\mathcal{Q}$, will, without causing confusion, be used to denote the corresponding edge in any of these representatives.

Let $\mathcal{Q}$ be an arbitrary sequence in $\Lambda_0$, and let $N_\mathcal{Q}$ be any $2$-connected representative of $G(\mathcal{Q})$. Let $h_a$ be any internal Steiner edge of $G(\mathcal{Q})$. Consider the following step-by-step process that converts $N_\mathcal{Q}$ into another $2$-connected representative of $G(\mathcal{Q})$, say $M(a)$, by relocating terminal endpoints of external Steiner edges. At the $i$-th step the endpoint of Steiner edge $h_i$ is relocated (if $h_i$ is an internal Steiner edge then nothing is done at this step). Since the resulting graph $M(a)$ is also a representative of $G(\mathcal{Q})$, the terminal endpoint of $h_i$ before and after step $i$ must belong to the same valid subset of $X$. The choice of new endpoint of $h_i$ is arbitrary, except in the following case:

Let $x_i$ be the terminal endpoint of $h_i$ in $N_\mathcal{Q}$, and let $x_i'$ be its endpoint in $M(a)$. Suppose that $h_i$ is incident to a terminal in a valid subset $X'$, where $X'$ is the interior of a degree two block path $H$ of $G_\mathrm{UN}$. Let $P$ be a shortest path in $N_\mathcal{Q}-h_i$ such that an end-edge of $P$ is $h_i$; $P$ connects $x_i$ to a distinct terminal $y'$ in $H$; and the interior of $P$ does not intersect $H$. Let $B_1,...,B_{p_H}$ be an orientation of the blocks of $H$ such that the index of the block containing $x_i$ is not smaller than the index of the block containing $y'$. Then $x_i'$ is selected to lie anywhere in $X'$ such that index of the block of $H$ containing $x_i'$ is maximised and the resulting graph is still $2$-connected.

\noindent\textbf{Observation} If $M_i(a)$ is the graph after step $i$ then $M_{i}(a)-h_a$ contains at least as many blocks as $M_{i+1}(a)-h_a$.

\begin{figure}[htb]
  \begin{center}
    \subfigure[]{\label{figEgs24a}\includegraphics[scale=0.5]{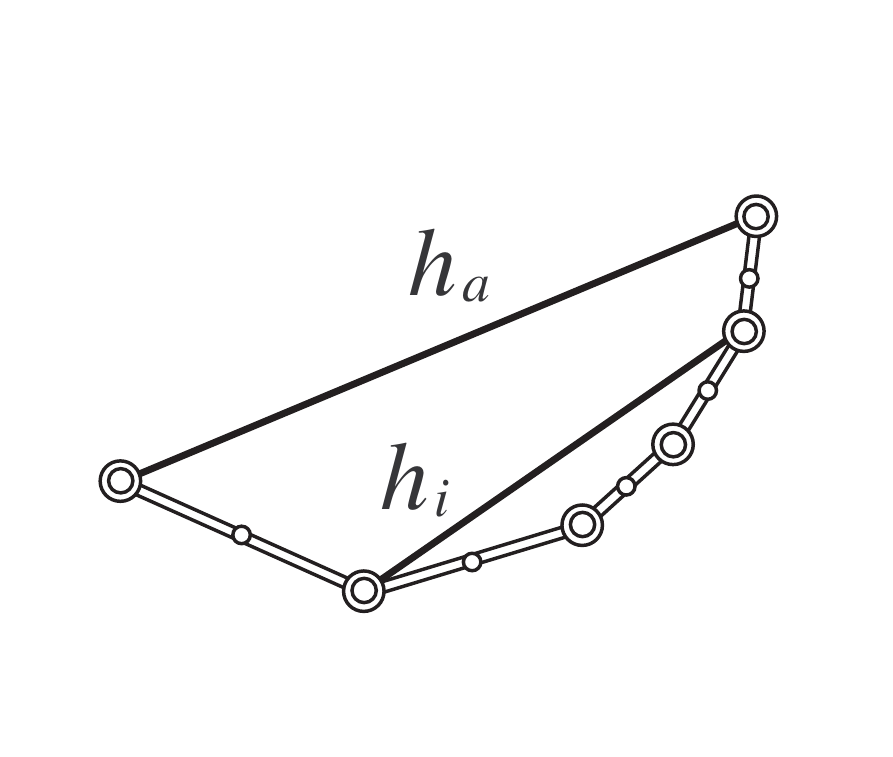}}
    \subfigure[]{\label{figEgs24b}\includegraphics[scale=0.5]{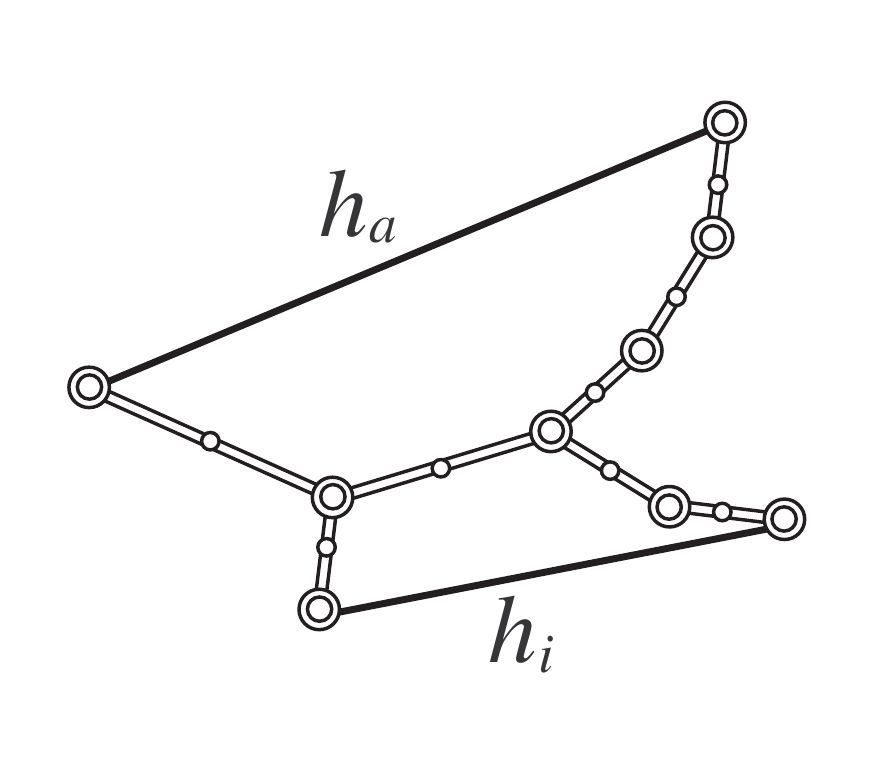}}
    \subfigure[]{\label{figEgs24c}\includegraphics[scale=0.5]{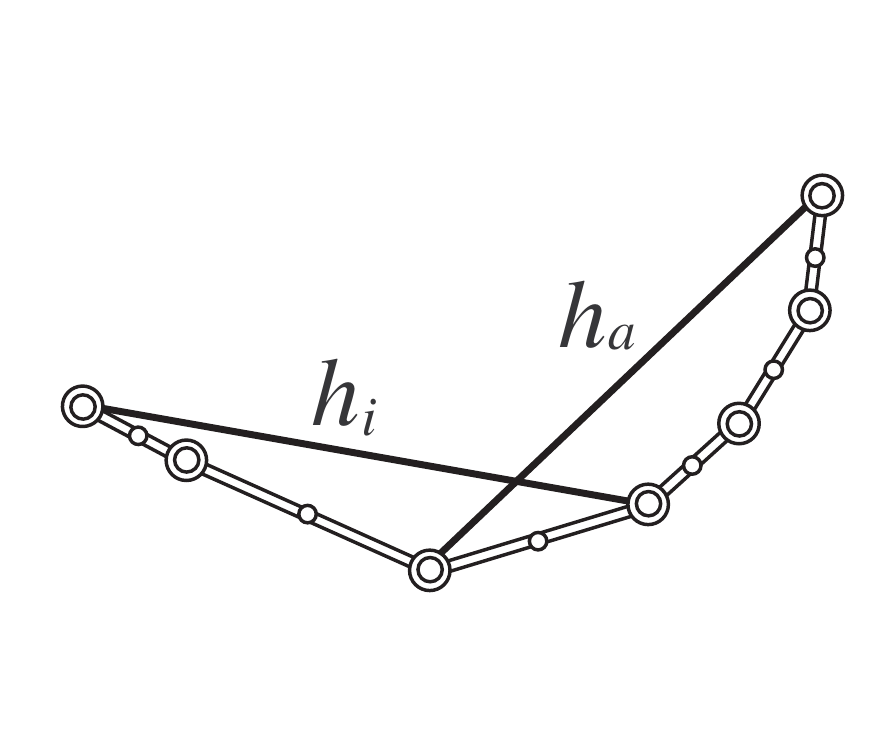}}
    \subfigure[]{\label{figEgs24d}\includegraphics[scale=0.5]{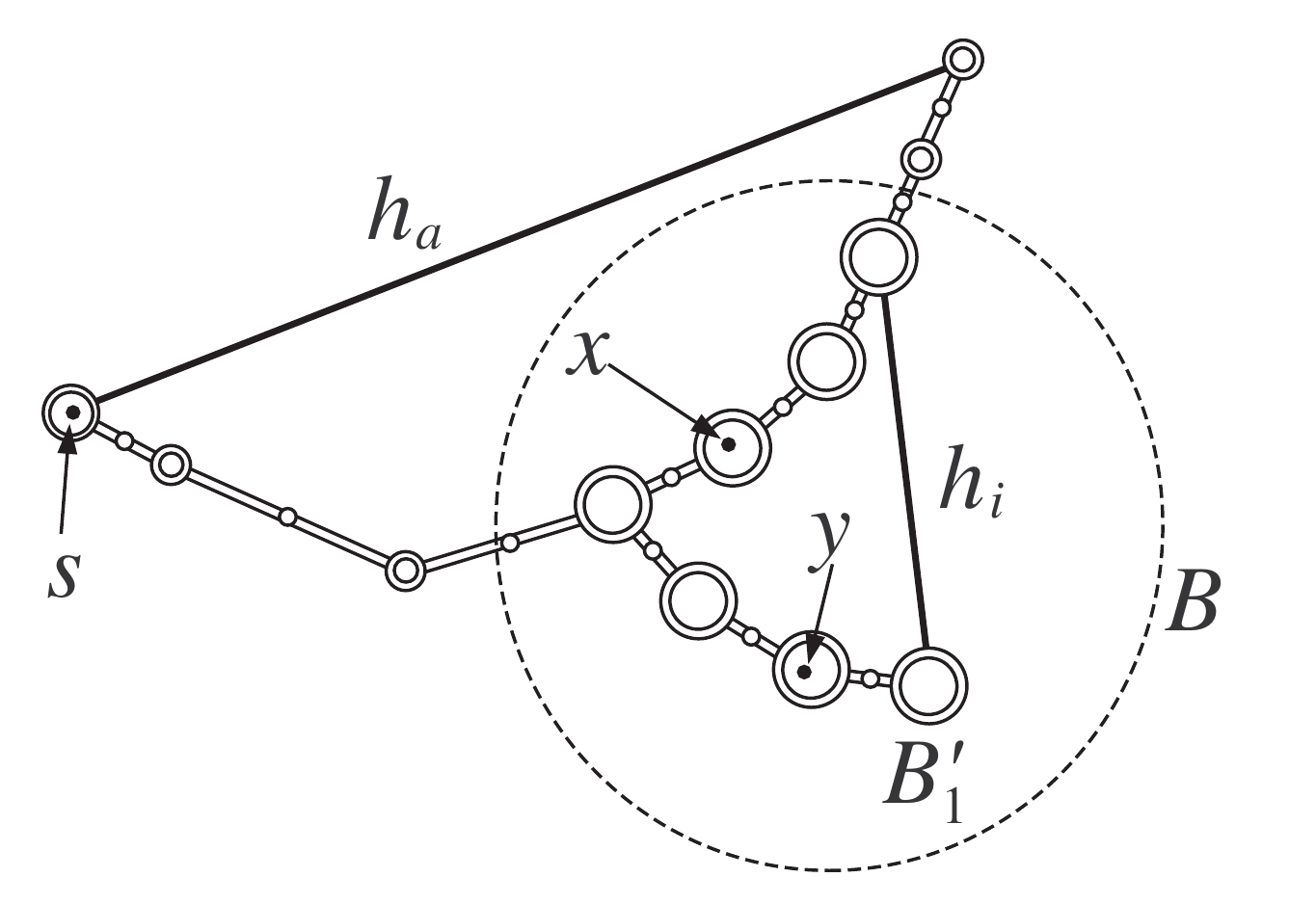}}
  \end{center}
  \caption{}
  \label{figEgs24}
\end{figure}

\begin{lemma}\label{keepOrd}Let $s$ be an endpoint of an internal Steiner edge $h_a$. Let $x,y$ be any two terminals such that there exists a path $P'$ in $N_\mathcal{Q}-h_a$ connecting $s$ and $y$, and such that $x$ is an interior vertex of $P'$. Then there exists a path $P''$ connecting $s$ and $y$ in $M(a)-h_a$ such that $P''$ also contains $x$ in its interior.
\end{lemma}
\begin{proof}
We show that this path property is preserved at each step of the process which converts $N_\mathcal{Q}$ to $M(a)$. Consider step $i$, where the terminal endpoint of $h_i$ is relocated. Once again, let $x_i$ be the terminal endpoint of $h_i$ in $N_\mathcal{Q}$, and let $x_i'$ be its endpoint in $M(a)$. There are four different cases we need to consider, represented by Figure \ref{figEgs24}. The case of Figure \ref{figEgs24a} is not possible, since it is assumed that all Steiner edges are critical. The lemma holds for Figure \ref{figEgs24b} since, by Lemma \ref{bcfIn}, $x_i$ and $x_i'$ lie in the same leaf-block of $M_i(a)-h_a-h_i$. We only consider the case from Figure \ref{figEgs24d} since the reasoning for Figure \ref{figEgs24c} is similar.

We consider the subcase when $x$ and $y$ are located as in Figure \ref{figEgs24d}; the remaining subcases are similar. Observe that $h_i$ is an edge of $P'$, and it occurs on the sub-path of $P'$ connecting $x$ and $y$. Let $B$ be the block of $M_i(a)-h_a$ containing $h_i$, and Let $B_1',...,B_{q}'$ be the block path $B-h_i$, oriented so that the index of the block containing $y$ is no larger than that of the block containing $x$.

Suppose that $x$ is contained in block $B_{j_x}'$. Then the lemma holds if $x_i'$ is in block $B_j'$, where $j\geq j_x$, but $x_i'$ is not the cut-vertex of $B_{j_x}'$ shared with $B_{j_x-1}'$; or if $x_i'$ is not in $B$ (note of course that in this case the block containing $x_i'$ must have an index larger than that of $B$, where the orientation is such that the first block of $M_i(a)-h_a$ contains $s$.) But these conditions hold, for otherwise $M_{i+1}(a)-h_a$ would have more blocks than $M_i(a)-h_a$, which contradicts the above observation.
\end{proof}

\begin{figure}[htb]
  \begin{center}
    \includegraphics[scale=0.5]{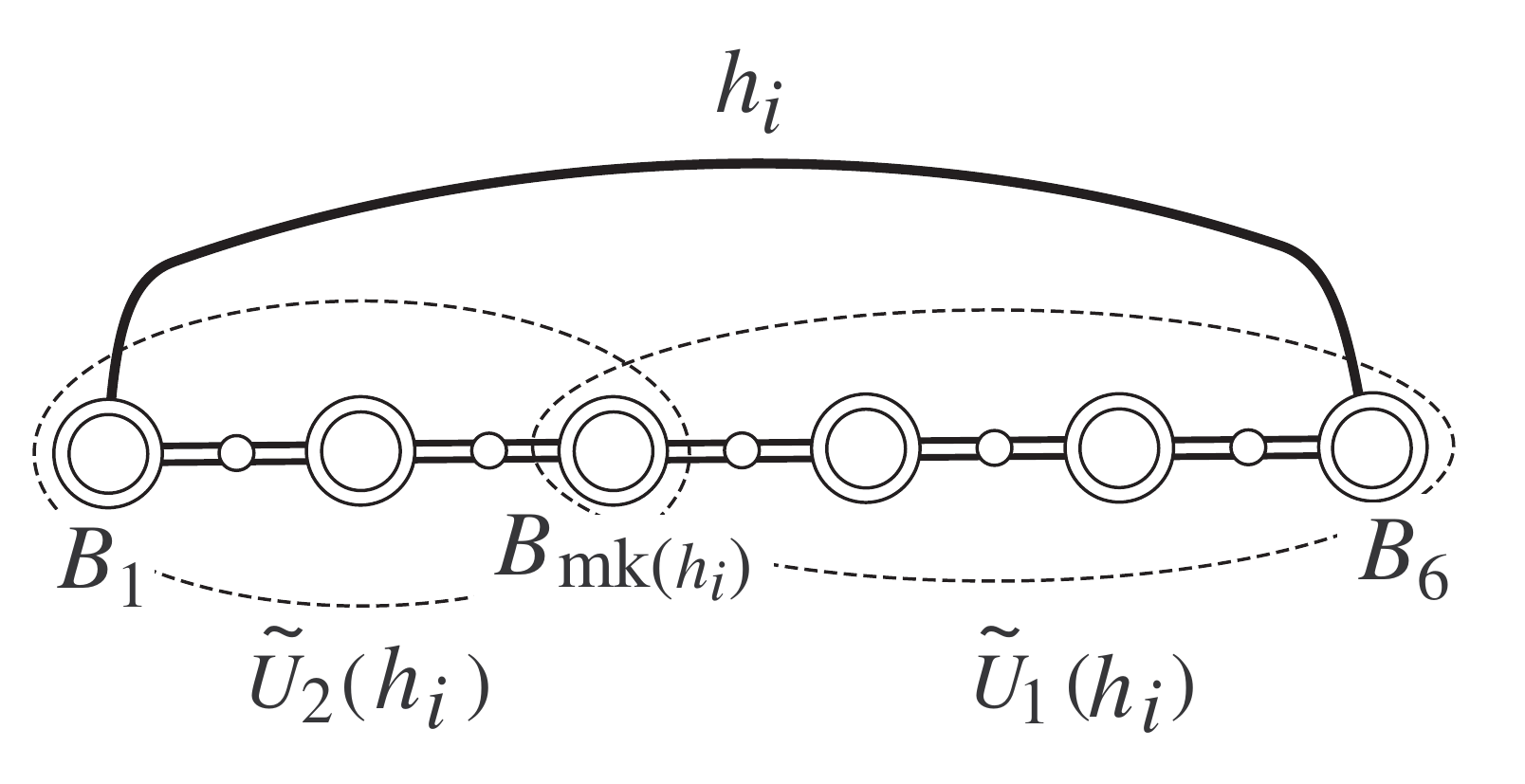}\\
  \end{center}
  \caption{An example of $\widetilde{U}_1(h_i)$ and $\widetilde{U}_2(h_i)$ when $M(Z,i)-h_i$ has $p=6$ blocks, and $V_1=\emptyset$}
  \label{figEgs10}
\end{figure}

We are now ready to formally define ``markers". Recall the recall $\preccurlyeq$ defined before Lemma \ref{isLink}. Let $\mathcal{Q}\in \Lambda$ and let $\mathcal{E}_\mathcal{Q}=\{W_j\}$ be the set of all non-empty sets of internal Steiner edges of $M_\mathcal{Q}$ containing at most $\Delta k-\vert E_\mathrm{SE}(M_\mathcal{Q})\vert$ elements. For every $h_i=s_is_i'\in W_j$, we construct (in at most $O(n^2)$ time) the block path $B_1,...,B_{p_Z}$ of $M(i)-h_i$, where $s_i\in B_1$ and $s_i'\in B_{p_Z}$. For any $h_i\in W_j$ it is assumed that there exist terminals $y,y'$ with $y'\preccurlyeq y$ with respect to the path $B_1,...,B_{p_Z}$ (note therefore that $p_Z>2$); any $W_j$ that does not have this property is removed from $\mathcal{E}_\mathcal{Q}$. Let $\mathrm{mk}({h_i})$ (the \textit{marker} for $h_i$) be a variable member of $\{2,...,p_Z-1\}$, let $\widetilde{U}_1(h_i)=B_{\mathrm{mk}({h_i})}\cup...\cup B_{p_Z}-B_{\mathrm{mk}({h_i})-1}$, and let $\widetilde{U}_2(h_i)=B_{1}\cup...\cup B_{\mathrm{mk}({h_i})}-B_{\mathrm{mk}({h_i})+1}$ (see Figure \ref{figEgs10}).

For a given set of markers $\mathcal{I}_j$ for the edges of $W_j$, let $\Phi(\mathcal{Q},\mathcal{I}_j)$ be the set of Steiner endpoint sequences constructed recursively from ${\mathcal{Q}}$ as follows. Let $T$ be the recursion tree with root ${\mathcal{Q}}_0={\mathcal{Q}}$. For every $i\geq 0$, at any $i$-th level node $w_i$ of $T$ we replace the pair $(s_{i+1},\{s_{i+1}'\})$ in $ {\mathcal{Q}}_{i}$ by $(s_{i+1}, {U}_1(h_{i+1})\cap C)$ and $(s_{i+1}', {U}_2(h_{i+1})\cap C')$, where $C,C'$ are any components of $G_\mathrm{UN}$ and where $h_{i+1}\in W_j$. The children of $w_i$ correspond to the different possible choices of $C,C'$; each distinct choice resulting in a distinct sequence $ {\mathcal{Q}}_{i+1}$. Any choices such that $ {U}_2(h_{i+1})\cap C'=\emptyset$ or $ {U}_1(h_{i+1})\cap C=\emptyset$ are discarded. This recursive process of transforming $ {\mathcal{Q}}$ into the set $\Phi(\mathcal{Q},\mathcal{I}_j)$ is referred to as Function \texttt{MarkSES}. Since $k$ and the number of components of $G_\mathrm{UN}$ are constant, the time-complexity of Function \texttt{MarkSES} is at most $O(n^2)$.

If $\mathcal{C}$ is the sequence of $\vert W_j\vert$ pairs of components of $G_\mathrm{UN}$ that were chosen in the construction of some $\mathcal{Q}'\in\Lambda'$ from $\mathcal{Q}$ then use the notation $\phi(\mathcal{Q},\mathcal{I}_j,\mathcal{C}):=\mathcal{Q}'$. For any linked set $\{g_{i,1},g_{i,2}\}$ derived from some internal edge $h_i\in W_j$ we employ the notation $\mathrm{mk}(g_{i,1}):=\mathrm{mk}(g_{i,2}):=\mathrm{mk}(h_i)$ for the current value of the marker of $h_i$.

\begin{theorem}\label{mLink}Let $N^*$ be a cheapest network containing $G_\mathrm{UN}$, such that all Steiner edges and chord-paths of $N^*$ are critical. Then there exists $\mathcal{Q}\in\Lambda_0$, a set of markers $\mathcal{I}_j$, and $\widehat{\mathcal{Q}}^*\in \Phi(\mathcal{Q},\mathcal{I}_j)$ such that $N^*$ is a representative of $G(\widehat{\mathcal{Q}}^*)$.
\end{theorem}
\begin{proof}
Let $N\in \mathcal{N}_0$ and $E_\mathrm{I}$ be a set of internal Steiner edges of $N$ such that $N^*$ is a split of $N$ with respect to $E_\mathrm{I}$. Let $\mathcal{Q}\in \Lambda_0$ be a Steiner endpoint sequence for $N$, and construct $M(i)$ for every $i$. For any $h_i\in E_\mathrm{I}$, let $y_i,y_i'$ be the endpoints of the external Steiner edges of $N^*$ that result from splitting $h_i$. Suppose, without loss of generality that $y_i\preccurlyeq y_i'$. As a consequence of Lemma \ref{keepOrd}, there exists a marker $\mathrm{mk}(h_i)$ such that $y_i\in \widetilde{U}_1(h_i)$ and $y_i'\in \widetilde{U}_2(h_i)$. Since this is true for every $i$, Function \texttt{MarkSES} will construct a $\widehat{\mathcal{Q}}^*$ such that the theorem follows.
\end{proof}

\section{The $2$-Bottleneck algorithm}\label{final}
In this section we present the $2$-Bottleneck algorithm and prove its correctness. Besides employing Functions \texttt{BuildSES}, \texttt{MarkSES}, and \texttt{2Connect}, the $2$-Bottleneck algorithm also employs a new subroutine, Function \texttt{BinLink}, for dealing with linked sets. In turn, Function \texttt{BinLink} depends on the recursive Function \texttt{CalcOpt}, which we present next. In simple terms, Function \texttt{BinLink} employs a binary search on the markers for each linked set in order to find optimal locations for the markers with respect to a given $\mathcal{Q}\in\Lambda_0$.

\begin{algorithm}[H]
\SetAlgoLined
\NoCaptionOfAlgo
\KwIn{$\mathcal{Q}\in\Lambda_1$, $\mathcal{I},\mathcal{I}_\mathrm{max}$, and a sequence $\mathcal{C}$ of component pairs of $G_\mathrm{UN}$}
\tcc*[h]{\small{CalcOpt acts on global variables OptLength and $G_\mathrm{OPT}$ from Algorithm BinLink}, and therefore has no direct output}

Construct $\mathcal{Q}':= \phi(\mathcal{Q},\mathcal{I},\mathcal{C})$ using Function \texttt{MarkSES}\;
{
Construct $G_\mathrm{OPT}^2(\mathcal{Q}')$ using Function \texttt{2Connect}\;\label{thisline}
\If{$\ell_{\max}(G_\mathrm{OPT}^2(\mathcal{Q}'))< \mathrm{OptLength}$}
{
$G_\mathrm{OPT}:=G_\mathrm{OPT}^2(\mathcal{Q}')$\;
OptLength $:=\ell_{\max}(G_\mathrm{OPT}^2(\mathcal{Q}'))$\;
}
Let $T$ be a component of the Steiner topology of $G_\mathrm{OPT}^2(\mathcal{Q}')$ such that a longest Steiner edge of $G_\mathrm{OPT}^2(\mathcal{Q}')$ is in $T$\;\label{lineT}
\For{every external Steiner edge $e$ of $T$ such that $\{e,\hat{e}\}$ is a linked set created by Function \normalfont{\texttt{MarkSES}}}
{
{For any edge $f$ let $\mathrm{mk}(f)$ denote the marker of $f$ in $\mathcal{I}$ and let $\mathrm{mk}_\mathrm{max}(f)$ denote the marker of $f$ in $\mathcal{I}_\mathrm{max}$}\;

\If {\normalfont ($\mathrm{mk}(e)\neq \mathrm{mk}_\mathrm{max}(e)$) \textbf{and} ($\mathrm{mk}(\hat{e})\neq \mathrm{mk}_\mathrm{max}(\hat{e})$)}
{
\eIf{$\vert \mathrm{mk}(e)- \mathrm{mk}_\mathrm{max}(e)\vert = 1$}
{$I:=\{\mathrm{mk}(e),\mathrm{mk}_\mathrm{max}(e)\}$}
{Let $I$ be the singleton containing $\lfloor(\mathrm{mk}(e)+\mathrm{mk}_\mathrm{max}(e))/2\rfloor$}
\For{every $\mathrm{mk}'(e)\in I$}
{\label{lineF}
Let $\mathcal{I}^{\,\prime}$ be the marker set that results from $\mathcal{I}$ after replacing $\mathrm{mk}(e)$ and $\mathrm{mk}(\hat{e})$ by $\mathrm{mk}'(e)$\;\label{lined}
Let $\mathcal{I}_\mathrm{max}^{\,\prime}$ be the marker set that results from $\mathcal{I}_\mathrm{max}$ after replacing $\mathrm{mk}_\mathrm{max}(\hat{e})$ by $\mathrm{mk}(e)$\;
Call Function \texttt{CalcOpt} with input $\mathcal{Q},\mathcal{I}^{\,\prime},\mathcal{I}_\mathrm{max}^{\,\prime}$, and $\mathcal{C}$\;\label{lineb}
}
}
}

}
\caption{\textbf{Function} CalcOpt}
\end{algorithm}

\newpage

\begin{algorithm}[H]
\SetAlgoLined
\NoCaptionOfAlgo
\KwIn{A Steiner endpoint sequence $\mathcal{Q}\in\Lambda_0$}
\KwOut{A graph $G_\mathrm{OPT}$ which is a cheapest graph in $\{G_\mathrm{OPT}^2(\mathcal{Q}'):\mathcal{Q}'\in\Phi(\mathcal{Q},\mathcal{I}_j); \mathcal{I}_j \text{ is a marker set for the edges of $M_\mathcal{Q}$}\}$}

Let $\mathrm{OptLength} := \infty$\;

\For{every $W_j\in\mathcal{E}_\mathcal{Q}$}
{
\For{every sequence $\mathcal{C}$ of $|W_j|$ pairs of components of $G_\mathrm{UN}$}
{
Let $\mathcal{I}$ be the marker set such that for every $h_i\in W_j$ the marker for $h_i$ is a median of $\{2,...,p_Z-1\}$, where the block path of $M(Z,i)-h_i$ is $B_1,...,B_{p_Z}$\;\label{lineZ}
Let $\mathcal{I}_\mathrm{max}$ be the set of markers $\mathrm{mk}_\mathrm{max}(\cdot)$ such that, for every $i$, if $\{g_{i,1},g_{i,2}\}$ is the linked set that is to replace internal edge $h_i$ in Function \texttt{MarkSES}, where $g_{i,1}$ is incident to $B_1$, then $\mathrm{mk}_\mathrm{max}(g_{i,1}):= 1$ and $\mathrm{mk}_\mathrm{max}(g_{i,2}):= p_Z$\;
Call Function \texttt{CalcOpt} with input $\mathcal{Q},\mathcal{I},\mathcal{I}_\mathrm{max}, \mathcal{C}$\;
}
}
\caption{\textbf{Function} BinLink \label{binLink}}
\end{algorithm}

\begin{proposition}For any $\mathcal{Q}\in\Lambda_0$, Function \normalfont{\texttt{BinLink}} \textit{correctly computes $G_\mathrm{OPT}$, a cheapest graph in $\{G_\mathrm{OPT}^2(\mathcal{Q}'):\mathcal{Q}'\in\Phi(\mathcal{Q},\mathcal{I}_j); \mathcal{I}_j \text{ is a marker set for the edges of $M_\mathcal{Q}$}\}$}.
\end{proposition}
\begin{proof}
Since Function \texttt{BinLink} considers every $W_j$ and every $\mathcal{C}$, correctness will follow if we show that Function \texttt{CalcOpt} correctly finds a cheapest network with respect to fixed $W_j$ and $\mathcal{C}$. 
Let $G_1,...G_q$ be a maximal sequence of graphs constructed by consecutive calls of Line \ref{thisline} in Function \texttt{CalcOpt}. In other words, if $T_q$ is the tree constructed in Line \ref{lineT} in the same call of Function \texttt{CalcOpt} that constructs $G_q$, then either $T_q$ has no liked sets that were created by Function \texttt{MarkSES}, or, for every such linked set $\{e,\hat{e}\}$ with $e$ in $T_q$, we have $\mathrm{mk}^q(e)=\mathrm{mk}^q_\mathrm{max}(e)$ or $\mathrm{mk}^q(\hat{e})=\mathrm{mk}^q_\mathrm{max}(\hat{e})$. Let $\mathcal{I}(i),\mathcal{I}_{\max}(i)$ be the marker sets $\mathcal{I},\mathcal{I}_\mathrm{max}$ in the same call of Function \texttt{CalcOpt} that constructs $G_i$. For any $e$ we use the notation $\mathrm{mk}^i(e)$ and $\mathrm{mk}_{\max}^i(e)$ to refer to $e$'s marker in $\mathcal{I}(i),\mathcal{I}_{\max}(i)$ respectively. For any linked set $\{g_{i,1},g_{i,2}\}$ we refer to $g_{i,1}$ as a \textit{left edge} and $g_{i,2}$ as a \textit{right edge}.

We define the following property: Property $L_i$ is satisfied if there exists a sequence $G_1,...,G_i$ and a set of markers $\mathcal{I}^*$ (called \textit{optimal markers}) such that, for every linked set $\{e,\hat{e}\}$ created by Function \texttt{MarkSES} (where, without loss of generality, $e$ is a left edge and $\hat{e}$ is a right edge), the marker of $e$ in $\mathcal{I}^*$, say $\mathrm{mk}^*(e)$, satisfies $\mathrm{mk}_{\max}^i(e)\leq \mathrm{mk}^*(e)\leq \mathrm{mk}_{\max}^i(\hat{e})$, and such that a cheapest network with respect to $W_j$ and $\mathcal{C}$ has $\mathcal{I}^*$ as a marker set. We claim that the proposition will immediately follow if Property $L_q$ holds, where $G_1,...,G_q$ is maximal. Suppose first that $T_q$ contains no linked sets that were created by Function \texttt{MarkSES}. But then, since $\ell_\mathrm{max}(T_q)\leq \ell_\mathrm{max}(G')$ for all $G'\in\{G_\mathrm{OPT}^2(\mathcal{Q}'):\mathcal{Q}'\in\Phi(\mathcal{Q},\mathcal{I}_j); \mathcal{I}_j \text{ is a marker set for the edges of $M_\mathcal{Q}$}\}$, we may set $G_\mathrm{OPT}:=G_q$ and the proposition follows. If $T_q$ does contain linked sets that were created by Function \texttt{MarkSES} then
$\mathrm{mk}^q(e)=\mathrm{mk}^q_\mathrm{max}(e)$ or $\mathrm{mk}^q(\hat{e})=\mathrm{mk}^q_\mathrm{max}(\hat{e})$ for every such linked set $\{e,\hat{e}\}$. Therefore $\vert\mathrm{mk}^q_\mathrm{max}(\hat{e})-\mathrm{mk}^q_\mathrm{max}({e})\vert=1$ so that $\mathrm{mk}^*(e)\in\{\mathrm{mk}^q_\mathrm{max}(\hat{e}),\mathrm{mk}^q_\mathrm{max}({e})\}$. But both of these markers are considered in Line \ref{lineF} by the previous call of Function \texttt{CalcOpt} that moved the marker of $e$ or $\hat{e}$, and therefore the claim follows.

We now use induction on the $G_i$. Clearly the base case for $G_1$ holds since $\mathrm{mk}_{\max}^1(e)=1$ and $\mathrm{mk}_{\max}^1(\hat{e})=p_Z$ for any linked set $e,\hat{e}$, where $p_Z$ is defined as in Line \ref{lineZ} of Function \texttt{BinLink}. Suppose Property $L_i$ holds for some $1\leq i<q$ and suppose that $G_{i+1}$ is not a cheapest network with respect to $W_j,\mathcal{C}$ (for otherwise Property $L_{i+1}$ holds and the proposition follows immediately). Let $T$ contain a longest edge of $G_{i}$. If $T$ contains no linked sets then $G_{i+1}$ is a cheapest network with respect to $W_j,\mathcal{C}$. We claim therefore that, for some $\mathcal{I}^*$ satisfying Property $L_i$, there exists an edge $e$ of $T$ contained in a linked set such that $\mathrm{mk}^*(e)\leq \mathrm{mk}^{i}(e)$ if $e$ is a left edge or $\mathrm{mk}^*(e)\geq \mathrm{mk}^{i}(e)$ if $e$ is a right edge. For otherwise, by the Monotonicity Property, $\ell_{\max}(T)\leq \ell_{\max}(G^*)$, where $G^*$ is an optimal network with respect to $\mathcal{I}^*$ (i.e., $G^*$ is a cheapest network with respect to $W_j$ and $\mathcal{C}$), which would imply that $G_{i+1}$ is a cheapest network with respect to $W_j$ and $\mathcal{C}$. Therefore the claim holds and, since Function \texttt{CalcOpt} considers all linked sets of $T$, it follows that Property $L_{i+1}$ holds for $\mathcal{I}^*$. Therefore, by induction, Property $L_q$ holds and the proposition follows.
\end{proof}

We now present our $2$-Bottleneck algorithm. For any $d$ in some interval of integers $L=[d_1,...,d_2]$, an \textit{upper} median (respectively \textit{lower} median) of $L$ with respect to $d$ is a median of $[d,...,d_2]$ (respectively $[d_1,...,d]$). Recall that $\Delta=7$ if we are working in the $L_1$ or $L_\infty$ norms, and $\Delta=5$ otherwise.\\

\begin{algorithm}[H]
\SetAlgoLined
\NoCaptionOfAlgo
\KwIn{A set $X$ of vertices in an $L_p$ plane, and a positive integer $k$}
\KwOut{A globally optimal network $N^*$ spanning $X$ and at most $k$ Steiner points}
Let $L$ be a non-decreasing sequence of distances between pairs of vertices in $X$\;
Let $d$ be a median of $L$\;
Let $\mathrm{FoundOpt}:=0$ and let $\mathrm{OptLength2}:=\infty$\;
\While{$\mathrm{FoundOpt}=0$}
{
Construct the block-cut forest of $G_\mathrm{UN}:=[K]_d$\;
\eIf{$b(G_\mathrm{UN})>\Delta k$}
{
Let $d$ be an upper median of $L$ with respect to $d$\;
}
{
Run Function \texttt{BuildSES} on $G_\mathrm{UN}$ to get $\Lambda_0$\;
Store the potential cuts for each $G(\mathcal{Q})$, where $\mathcal{Q}\in \Lambda_0$, for later use by Function \texttt{2Connect}\;
\For{every $\mathcal{Q}\in\Lambda_0$}
{
Run Function \texttt{BinLink} with input $\mathcal{Q}$ and output $G_\mathrm{OPT}$\;
\If{$\ell_\mathrm{max}(G_\mathrm{OPT})<\mathrm{OptLength2}$}
{
$\mathrm{OptLength2}:=\ell_\mathrm{max}(G_\mathrm{OPT})$\;
$N^*:=G_\mathrm{OPT}$\;
}
}
\eIf{$\mathrm{OptLength2}\leq d$}
{
Let $d$ be a lower median of $L$ with respect to $d$. If $d$ has been considered before then let $\mathrm{FoundOpt}:=1$\;
}
{
Let $d$ be an upper median of $L$ with respect to $d$. If $d$ has been considered before then let $\mathrm{FoundOpt}:=1$\;
}

}
}
\caption{\textbf{$2$-Bottleneck Algorithm}\label{main}}
\end{algorithm}

\begin{theorem}The $2$-Bottleneck Algorithm correctly constructs a globally optimal network spanning $X$ and most $k$ Steiner points. The run time is $O(n^2\log^{\frac{7k}{2}+1}n)$ in $L_1$ and $L_\infty$, and $O(n^k\log^{\frac{5k}{2}}n)$ for all other $L_p$ planes.
\end{theorem}
\begin{proof}
We only still need to show that a globally optimal underlying network is found by the $2$-Bottleneck algorithm. Let $G_0,G_1$ be two underlying network such that $G_0$ is a subgraph of $G_1$. Let $N_0^*$ be a cheapest network containing $G_0$ and let $N_1^*$ be a cheapest network containing $G_1$. Then, similarly to the Monotonicity Property, the length of a longest Steiner edge in $N_1^*$ is no longer than the length of a longest Steiner edge in $N_0^*$. Therefore the $2$-Bottleneck algorithm correctly performs a binary search on the elements of $L$, which, in turn, are used to construct the underlying networks. Therefore, since Function \texttt{BinLink} is correct, the $2$-Bottleneck algorithm is also correct.

The complexity of the loop in Line 4 is $O(\log n)$, since a binary search is performed on $L$. Function BuildSES runs in $O(n^2)$ time and storing the potential cuts takes $O(n^2)$ time. Observe that since $G_\mathrm{UN}$ contains at least two leaf blocks, there can be at most ${\Delta k}/{2}-1$ markers. Therefore Function \texttt{BinLink} runs in a time of $O(\log n\times\log^{\frac{\Delta k}{2}-1} n\times f(n,k))$, where $f(n,k)$ is the complexity of finding $G_\mathrm{OPT}(\mathcal{Q})$ (provided in Theorem \ref{consOp}). Therefore the theorem follows.
\end{proof}

\section{Conclusion}
In this paper we present the first exact polynomial time algorithm for constructing optimal bottleneck $2$-connected $k$-Steiner networks in $L_p$ planes when $k$ is constant. The algorithm runs in $O(n^2\log^{\frac{7k}{2}+1}n)$ steps in $L_1$ and $L_\infty$, and in $O(n^k\log^{\frac{5k}{2}}n)$ steps for all other $L_p$ planes. This significantly extends and generalises the results of Bae et al. \cite{bae1} and Brazil et al. \cite{brazil}, which solve the $1$-connected case, and Brazil et al. \cite{brazil2} which solves the $2$-connected case for $k\leq 2$.

\end{document}